\documentclass[reqno]{amsart}
\usepackage{amsmath,amsthm,amsfonts, amssymb,amscd,enumitem}
\usepackage[all]{xy}
\numberwithin{equation}{section}
  \newtheorem{thm}{Theorem}[section]
  \newtheorem{lem}[thm]{Lemma}
  \newtheorem{prop}[thm]{Proposition}
  \newtheorem{cor}[thm]{Corollary}
  \theoremstyle{definition}
  \newtheorem{defn}[thm]{Definition}
  \newtheorem{exm}[thm]{Example}
  \newtheorem{rmk}[thm]{Remark}
  \newtheorem{open}[thm]{Problem}

 \newcommand\ra{\rightarrow}

\newcommand{\lex}{\,\overrightarrow{\times}\,}

 \newcommand\mI{\mathcal{I}}

 \newcommand\s{\subseteq}

 \newcommand\B{\mathrm{B}}
 \newcommand\C{\mathrm{C}}

\newcommand{\id}{\mbox{\rm Id}}
\newcommand\spec{\mathrm{Spec}}
\newcommand{\Aut}{\mbox{\rm Aut}}
\newcommand{\semid}{\,\overrightarrow{\rtimes}_{\phi}\,}

 \numberwithin{equation}{section}
\def\iff{if and only if }
\def\im{ \mathrm{Im} }

\let\Right\right
\let\Left\left
\makeatletter
\def\right#1{\Right#1\@ifnextchar){\!\right}{}}
\def\left#1{\Left#1\@ifnextchar({\!\left}{}}

\begin{document}
\title[Some results on pseudo MV-algebras with square roots]{Some results on pseudo MV-algebras with square roots}
\author[Anatolij Dvure\v{c}enskij and Omid Zahiri]{Anatolij Dvure\v{c}enskij$^{^{1,2,3}}$, Omid Zahiri$^{^{1,*}}$}

\date{}%
\thanks{Corresponding Author: Omid Zahiri}
\address{$^1$Mathematical Institute, Slovak Academy of Sciences, \v{S}tef\'anikova 49, SK-814 73 Bratislava, Slovakia}
\address{$^2$Palack\' y University Olomouc, Faculty of Sciences, t\v r. 17. listopadu 12, CZ-771 46 Olomouc, Czech Republic}
\address{$^3$ Depart. Math., Constantine the Philosopher University in Nitra, Tr. A. Hlinku 1, SK-949 01 Nitra, Slovakia}
\email{dvurecen@mat.savba.sk, zahiri@protonmail.com}
\thanks{}



\begin{abstract}
The paper provides a study of pseudo MV-algebras with square roots.
We introduce different notions of a square root on a pseudo MV-algebra, and present their main properties.
We show that the class of pseudo-MV-algebras with square roots is a proper subvariety of the variety of pseudo MV-algebras. Then, we define a strict square root to classify the class of pseudo MV-algebras with square roots.
We found a relationship between strongly atomless pseudo MV-algebras and strict pseudo MV-algebras.
Finally, we investigate square roots on representable symmetric pseudo MV-algebras, and we present a complete
characterization of a square root and a weak square root on a representable symmetric pseudo MV-algebra using addition in a unital $\ell$-group. Some interesting examples are provided.
\end{abstract}

\date{}

\maketitle

\section{Introduction }

A study of many valued algebras started by Chang, \cite{Cha1,Cha2}. He introduced MV-algebras as
an algebraic counterpart of many-valued reasoning. It was recognized that every interval in
a unital Abelian $\ell$-group gives an example of
MV-algebras. The principal result of the theory of MV-algebras is a representation theorem by
Mundici \cite{Mun} saying that there is a categorical equivalence between the category of
MV-algebras and the category of unital Abelian $\ell$-groups. Today's theory of MV-algebras is
very deep and has many interesting connections with other parts of mathematics with many
important applications to different areas of the research.

In the last period, two equivalent non-commutative generalizations of MV-algebras appeared.
One was  introduced in \cite{georgescu} as pseudo MV-algebras and the other one in \cite{Rac} as generalized
MV-algebras. They form an algebraic counterpart of the non-commutative \L ukasiewicz logic. The principal representation result on pseudo MV-algebras, \cite{Dvu2}, says that there is a categorical equivalence between the category of unital $\ell$-groups not necessarily Abelian and the category of pseudo MV-algebras. It generalizes a similar result by Mundici \cite{Mun} for MV-algebras. Pseudo MV-algebras were studied in many papers, see, e.g. \cite{GaTs, YaRu}. They were generalized to non-commutative BL-algebras, and fuzzy logic connected with these non-commutative algebraic structures was investigated in \cite{Haj}.

Square roots on MV-algebras were originally introduced in \cite{Hol}, and they are valuable tools for algebraic structures with non-idempotent binary operations.
They can provide us useful information about the elements and some properties of these types of algebras,
and they help us to build new operations from binary operations according to our needs.
For this reason, square roots have many applications in the theory of semigroups, groups and rings.

In \cite{Hol}, H{\"o}hle also attempted to explore the significance of square roots in the class of residuated lattices, integral,  commutative, residuated $\ell$-monoids. He classified the class of MV-algebras with square root and  presented a decomposition for each MV-algebra as a direct product of a Boolean algebra and a strict MV-algebra.
He also found the connection between the concepts of divisibility and objectivity and strictness on complete MV-algebras.
His study showed that a complete MV-algebra with square roots is either a complete Boolean algebra
or a Boolean valued model of the real unit interval viewed as an MV-algebra or a product of both.
Ambrosio \cite{Amb} continued to study square roots on MV-algebras. She introduced the concept of 2-atomless MV-algebras and showed
that strict MV-algebras are equivalent to 2-atomless MV-algebras. Then she proved that each strict MV-algebra contains a
copy of the MV-algebra of all rational dyadic numbers. Finally, she investigated the homomorphic image of MV-algebra with square roots.

B\v{e}lohl\'avek studied the class of all residuated lattices with square roots and showed that it is an equational class.
We refer to \cite{Hol,Amb,NPM,331} for more details about square roots on MV-algebras.

The main goals of this paper are to provide a study on pseudo MV-algebras with square roots and present their main properties:

\begin{itemize}
\item Define square roots and weak square roots that in the case of MV-algebras coincide.

\item Present strict square roots on pseudo MV-algebras.

\item Decompose square roots into two parts, one as the identity on the Boolean part and one as a strict square root.

\item Study strongly atomless pseudo MV-algebras and find their relationship with strict pseudo MV-algebras.

\item Represent square roots by addition in unital $\ell$-groups on representable symmetric pseudo MV-algebras.

\item  Give examples of weak square roots on totally ordered pseudo MV-algebras that are not a square root.
\end{itemize}

The paper is organized as follows.
Section 2 contains basic definitions, properties, and results about pseudo MV-algebras that will be used in the next sections. Section 3 introduces the notions of a square root and of a weak square root on a pseudo MV-algebra, and their main properties are obtained. It is proved that the class of pseudo MV-algebras with square roots is a proper subvariety of the variety of pseudo MV-algebras. The image of a square root on a zero element of pseudo MV-algebras plays an important role. In Section 4, we define the concept of a strict square root and classify the class of pseudo MV-algebras with square roots. We prove that for each pseudo MV-algebra $M$, only one of the following statements holds: (1) $M$ is a Boolean algebra; (2) $M$ is a strict  pseudo MV-algebra; (3) $M$ is isomorphic to the direct product of a Boolean algebra and a strict pseudo MV-algebra.
Using unital $\ell$-groups, Section 5 investigates any square root on representable symmetric pseudo MV-algebras.
Since any MV-algebra is representable and symmetric, these results also hold for MV-algebras. Finally, Section 6 provides two examples of weak square roots on totally ordered pseudo MV-algebras that are not square roots.

\section{Preliminaries}

In this section, we gather some preliminary results about generalized Boolean algebras and pseudo MV-algebras, which will be needed in the following sections. For more details about pseudo MV-algebras, we refer to \cite{georgescu,Dvu1,DvuS,Dvu2}.

\begin{defn}\label{de:PMV} \cite{georgescu}
A {\it pseudo MV-algebra} is an algebra $(M;\oplus,^-,^\sim,0,1)$ of type $(2,1,1,0,0)$ such that the following axioms hold
for all $x,y,z\in M$,

(A1) $x\oplus(y\oplus z)=(x\oplus y)\oplus z$,

(A2) $x\oplus 0=0\oplus x=x$,

(A3) $x\oplus 1=1\oplus x=1$,

(A4) $1^{-}=1^{\sim}=0$,

(A5) $(x^{-}\oplus y^{-})^{\sim}=(x^{\sim}\oplus y^{\sim})^{-}$,

(A6) $x\oplus (x^{\sim}\odot y)=y\oplus (y^{\sim}\odot x)=(x\odot y^{-})\oplus y=(y\odot x^{-})\oplus x$,

(A7) $x\odot (x^{-}\oplus y)=(x\oplus y^{\sim})\odot y$,

(A8) $(x^{-})^{\sim}=x$, \\
where
$$
x\odot y=(y^{-}\oplus x^{-})^{\sim}.
$$
If the operation $\oplus$ is commutative, equivalently, $\odot$ is commutative, then $M$ is an MV-algebra, and of course $x^-=x^\sim$ for each $x\in M$; we denote it then by $x':=x^-$. For more about MV-algebras, see \cite{CDM}.

It is well-known that every pseudo MV-algebra is a bounded distributive lattice, \cite{georgescu}, where axiom (A6) defines $x\vee y$ and (A7) describes $x\wedge y$.

Pseudo MV-algebras are intimately connected with unital $\ell$-groups. We recall that a {\it lattice-ordered group} ($\ell${\it-group}) is an algebra $(G;\vee,\wedge,+,-,0)$ such that $(G;\vee,\wedge)$ is a lattice, $(G;+,-,0)$ is a group, and $+$ is an order-preserving map. We denote by $G^+=\{g\in G\mid g\ge 0\}$. An element $u\in G^+$ is said to be a {\it strong unit} if given $g\in G$, there is an integer $n\ge 1$ such that $g\le nu$. A couple $(G,u)$, where $G$ is an $\ell$-group with a fixed strong unit $u$, is said to be a {\it unital $\ell$-group}.

If  $(G;\vee,\wedge,+,-,0)$ is an $\ell$-group, given $0\leq u$, we define an interval $[0,u]:=\{x\in G\mid 0\leq x\leq u\}$. The interval $[0,u]$ can be converted into a pseudo MV-algebra $\Gamma(G,u)=([0,u];\oplus,^-,^\sim,0,u)$ as follows: $x\oplus y = (x+y)\wedge u$, $x\odot y= (x-u+y)\vee 0$, $x^-=u-x$, and $x^\sim = -x+u$, $x,y \in [0,u]$. For more details on $\ell$-groups, consult with \cite{Dar,Fuc,Gla}.

Conversely, due to the principal result \cite{Dvu2}, every pseudo MV-algebra is isomorphic to $\Gamma(G,u)$ for some unital $\ell$-group $(G,u)$. The unital $\ell$-group $(G,u)$ is unique up to isomorphism of unital $\ell$-groups. Moreover, the category of unital $\ell$-groups is categorical equivalent to the category of pseudo MV-algebras. It generalizes a well-known result by Mundici, \cite{Mun}, who established it for MV-algebras.

Let $\mathcal{UG}$ be the category of unital  $\ell$-groups whose objects are unital $\ell$-groups $(G,u)$ and morphisms between objects are $\ell$-group homomorphisms preserving fixed strong units. We denote by $\mathcal{PMV}$ the category of pseudo MV-algebras whose objects are pseudo MV-algebras and morphisms are homomorphisms of pseudo MV-algebras. Consider the functor $\mathbf{ \Gamma}:\mathcal{UG}\to \mathcal{PMV}$ which is defined as follows
$$
\mathbf{\Gamma}(G,u)=(\Gamma(G,u);\oplus, ^-,^\sim,0,u),
$$
and if $h:(G,u)\to (H,v)$ is a morphism, then $\mathbf {\Gamma}(h)$ is the restriction of $ h$ onto the interval $[0,u]$.
On the other side, there is a functor from the category of $\mathcal{PMV}$ to $\mathcal{UG}$ sending a pseudo MV-algebra $M$ to a unital $\ell$-group $(G,u)$ such that $M \cong \Gamma(G,u)$ which is
denoted by $\mathbf{\Psi}:\mathcal{PMV}\ra \mathcal{UG}$. If $\phi: M_1\to M_2$ is a morphism of pseudo MV-algebras, then $\mathbf \mathbf{\Psi}(\phi)$ is a unique extension of $\phi$ onto a morphism of unital $\ell$-groups. It is possible to show that if $\phi:M_1\to M_2$ is an injective (surjective) homomorphism, so is $\mathbf{\Psi}(\phi)$.

\begin{thm}{\rm(\cite{Dvu2})}\label{functor}
The composite functors $\mathbf{\Gamma}\circ\mathbf{\Psi}$ and $\mathbf{\Psi}\circ\mathbf{\Gamma}$ are naturally equivalent to the identity functors of $\mathcal{PMV}$ and $\mathcal{UG}$, respectively. Therefore, the categories $\mathcal{PMV}$ and $\mathcal{UG}$ are categorically equivalent.
\end{thm}

Let $H$ and $G$ be $\ell$-groups. We define the lexicographic product $H\lex G$ as the group addition on the direct product $H\times G$ endowed with the lexicographic order $(h_1,g_1)\le (h_2,g_2)$ for $(h_1,g_1), (h_2,g_2)\in H\times G$ iff either $h_1<h_2$ or $h_1=h_2$ and $g_1\le g_2$. Then $H\lex G$ is a partially-ordered group that is an $\ell$-group iff $H$ is linearly ordered, see \cite[(d) p. 26]{Fuc}.

A pseudo MV-algebra $M$ is said to be {\it symmetric} if $x^\sim=x^-$ for each $x\in M$. We underline that if $M$ is symmetric, then $\oplus$ is not necessarily commutative. Indeed, let $G$ be a non-commutative $\ell$-group and $(H,1)$ be an arbitrary unital subgroup of $(\mathbb R,1)$. Then $(H\lex G,(1,0))$, where $H\lex G$ is the lexicographic product of $H$ and $G$, is a unital $\ell$-group. The pseudo MV-algebra $\Gamma(H\lex G,(1,0))$ is a symmetric pseudo MV-algebra that is not an MV-algebra.

We note that if $(G,u)$ is a unital $\ell$-group, the pseudo MV-algebra $\Gamma(G,u)$ is symmetric iff $u$ belongs to the center  $\C(G)=\{g \in G \mid g+h=h+g, \forall h \in G\}$.

In every pseudo MV-algebra $M$, we can define a partial operation $+$ in the following form: $x+y$ is defined in $M$ if $x\le y^-$, equivalently $y\le x^\sim$, equivalently $y\odot x=0$, and in either case, we set $x+y:= x\oplus y$, see \cite{Dvu2}.  We note that if $M=\Gamma(G,u)$, then $+$ is the restriction of the group addition of the group $G$ to the interval $[0,u]$.

For any integer $n\ge 0$ and any $x\in M$, we can define
\begin{align*}
0.x &=0, \quad  \text{and}\quad n.x=(n-1).x\oplus x,\quad n\ge 1,\\
0x&=0, \quad \text{and}\quad nx= (n-1)x+x,\quad n\ge 1,
\end{align*}
assuming $(n-1)x$ and $(n-1)x+x$ are defined in $M$.
An element $x\in M$ is called a {\em Boolean element} if $x\oplus x=x$. The set $\B(M)$ denotes the set of all Boolean elements of $M$, which is a Boolean algebra and a subalgebra of $M$. By a {\em degenerate} pseudo MV-algebra we mean a pseudo MV-algebra $(M;\oplus,^-,^\sim,0,1)$ with $0= 1$.
\end{defn}

In what follows, we will assume usually that $M$ is such that $0\ne 1$.

It is well-known if $(M;\oplus,^-,^\sim,0,1)$ is a pseudo MV-algebra, then, for each $a\in \B(M)$, the interval $([0,a];\oplus,^{-a},^{\sim a},0,a)$  is a pseudo MV-algebra, where $x^{-a}=x^-\wedge a$ and $x^{\sim a}=x^\sim\wedge a$, and the mapping $\phi_a:M\to M_a:=[0,a]$, defined by $\phi_a(x)=a\wedge x$, $x\in M$, is a surjective homomorphism of pseudo MV-algebras.

In each pseudo MV-algebra $(M;\oplus,^-,^\sim,0,1)$, we can define two additional binary operations $\to$ and $\rightsquigarrow$ by
$$
x\to y:=x^-\oplus y,\quad x\rightsquigarrow y:=y\oplus x^\sim.
$$

\begin{prop}\label{PMV-prop}
In any pseudo MV-algebra $(M;\oplus,^-,^\sim,0,1)$, the following properties hold for all $x,y,z\in M$:
\begin{itemize}[nolistsep]
\item[{\rm (i)}] $x\odot (x\to y)=x\wedge y\leq y$, and $(x\rightsquigarrow y)\odot x=x\wedge y\leq y$.
\item[{\rm (ii)}] $x\vee y=(x\to y)\rightsquigarrow y=(x\rightsquigarrow y)\to y=
(x\to y)\rightsquigarrow y=(x\rightsquigarrow y)\to y$.
\item[{\rm (iii)}] $(x\odot y)\to z=y\to (x\to z)$ and $(x\odot y)\rightsquigarrow z=x\rightsquigarrow (y\rightsquigarrow z)$.
\item[{\rm (iv)}] $x\to (x\odot y)=(x\to 0)\vee y$ and  $x\rightsquigarrow (y\odot x)=(x\rightsquigarrow 0)\vee y$.
\item[{\rm (v)}] $x\to (y\rightsquigarrow z)=y\rightsquigarrow (x\to z)$.
\item[{\rm (vi)}] $x\odot y\leq z \Longleftrightarrow y\leq x\to z \Longleftrightarrow x\leq y\rightsquigarrow z$.
\item[{\rm (vii)}] $x\wedge y=x\odot (x\to y)=y\odot (y\to x)=(y\rightsquigarrow x)\odot y=(x\rightsquigarrow y)\odot x$.
\item[{\rm (viii)}] $x\to (y\wedge z)=(x\to y)\wedge (x\to z)$ and $x\rightsquigarrow (y\wedge z)=(x\rightsquigarrow y)\wedge (x\rightsquigarrow z)$.
\item[{\rm (ix)}] $x\to (y\wedge z)=(x\to y)\odot ((x\wedge y)\to z)$ and $x\rightsquigarrow (y\wedge z)=((x\wedge y)\rightsquigarrow z)\odot (x\rightsquigarrow y)$.
\item[{\rm (x)}] $(x\wedge y)\to z=(x\to z)\vee (y\to z)$ and $(x\wedge y)\rightsquigarrow z=(x\rightsquigarrow z)\vee (y\rightsquigarrow z)$.
\item[{\rm (xi)}] If $x\leq y$, then $y\ra z\leq x\ra z$ and $y\rightsquigarrow z\leq x\rightsquigarrow z$.
\end{itemize}
\end{prop}

\begin{proof}
The proofs follow from the properties of pseudo MV-algebras which are stated in \cite{georgescu,Dvu1}.
\end{proof}
A non-empty subset $I$ of a pseudo MV-algebra $M$ is called an {\it ideal} if (1) for each $y\in M$, $y\leq x\in I$ implies that $y\in I$; (2) $I$ is closed under $\oplus$. An ideal $I$ of $M$ is said to be (i) {\it prime} if $x\wedge y \in I$ implies $x \in I$ or $y \in I$; (ii) {\it normal} if $x\oplus I=I\oplus x$ for any $x \in M$, where $x\oplus I:=\{x\oplus i \mid i \in I\}$ and $I\oplus x =\{i\oplus x \mid i \in I\}$.
We recall that an ideal $I$ is normal \iff given $x,y \in M$,  $x\odot y^-\in I$ \iff $y^\sim \odot x \in I$, \cite[Lem 3.2]{georgescu}. We note that the kernel of any homomorphism of pseudo MV-algebras is a normal ideal.

The set of prime ideals of $M$ is denoted by $\spec(M)$. Equivalent conditions, \cite[Thm 2.17]{georgescu}, for an ideal $I$ to be prime are as follows:
\begin{itemize}[nolistsep]
\item[{\rm (P1)}]  $x\odot y^- \in I$ or $y\odot x^-$ for all $x,y \in M$.
\item[{\rm (P2)}] $x\odot y^\sim \in I$ or $y\odot x^\sim$ for all $x,y \in M$.
\end{itemize}
A one-to-one relationship exists between congruences and normal ideals of a pseudo MV-algebra, \cite[Cor. 3.10]{georgescu}: If $I$ is a normal ideal of a pseudo MV-algebra, then the relation $\sim_I$, defined by $x\sim_I y$ \iff $x\odot y^-, y\odot x^-\in I$, is a congruence,
and $M/I$ with the following operations is a pseudo MV-algebra, where $M/I=\{x/I\colon x\in M\}$ and $x/I$ is the equivalence class containing $x$.
\begin{eqnarray*}
x/I\oplus y/I=(x\oplus y)/I,\quad (x/I)^-=x^-/I,\quad (x/I)^\sim=x^\sim/I,\quad 0/I,\quad 1/I.
\end{eqnarray*}
Conversely, if $\theta$ is a congruence on $M$, then $I_\theta =\{x \in M \mid x\sim 0\}$ is a normal ideal such that $\sim_{I_\theta}=\sim$.

A pseudo MV-algebra $M$ is said to be {\it representable} if $M$ is a subdirect product of a system of linearly ordered pseudo MV-algebras.
By \cite[Prop. 6.9]{DvuS}, $M$ is representable \iff $a^\bot=\{x \in M\mid x\wedge a =0\}$ is a normal ideal of $M$ for each $a \in M$. Moreover, $M=\Gamma(G,u)$ is representable iff $G$ is a representable $\ell$-group, see \cite[Prop. 6.9]{DvuS}.


\section{Square Roots on Pseudo MV-algebras}

The aim of the section is to define square roots on pseudo MV-algebras and present their basic properties. We define a square root and a weak square root on a pseudo MV-algebra. These notions coincide for MV-algebras, but for pseudo MV-algebras they can be different.

A square root on MV-algebras was originally defined in \cite{Hol}:

\begin{defn}\cite{Hol}
Let $M$ be an MV-algebra. A mapping $s:M\to M$ is called a {\em square root} if it satisfies the following conditions:
\begin{enumerate}[nolistsep]
\item[(i)] for all $x\in M$, $s(x)\odot s(x)=x$;
\item[(ii)] for all $x,y\in M$, $y\odot y\leq x$ implies $y\leq s(x)$.
\end{enumerate}
\end{defn}

In \cite[Prop 2.11(xxi)]{Hol}, it was shown that $s(x')=s(x)\to s(0)$. However, for pseudo MV-algebras, we need a stronger definition of the square root.

\begin{defn}\label{3.1}
Let $(M;\oplus,^-,^\sim,0,1)$ be a pseudo MV-algebra. A mapping $r:M\to M$ is said to be (i) a {\em square root} if it satisfies the following conditions:
\begin{itemize}
\item[{\rm (Sq1)}] for all $x\in M$, $r(x)\odot r(x)=x$;

\item[{\rm (Sq2)}] for each $x,y\in M$, $y\odot y\leq x$ implies $y\leq r(x)$;

\item[{\rm (Sq3)}] for each $x\in M$, $r(x^-)= r(x)\to r(0)$ and $r(x^\sim)=r(x)\rightsquigarrow r(0)$,
\end{itemize}
and (ii) a {\it weak square root} if it satisfies only (Sq1) and (Sq2). For MV-algebras these notions coincide but for pseudo MV-algebras they could be different as we show in Examples \ref{example}--\ref{example1} below.

If $M$ is degenerate, then $r=\id_M$ is trivially both a square root and a weak square root.

A pseudo MV-algebra $(M;\oplus,^-,^\sim,0,1)$ has {\em square roots} ({\it weak square roots}) if there exists a square root (weak square root) $r$ on $M$. Sometimes, one writes $x^{1/2}$ for square roots $r(x)$, see e.g. \cite{Hol}.
A {\em standard} square root is a square root (weak square root) $r$ satisfying the following additional condition:

\begin{itemize}
\item[{\rm (Sq4)}] for all $x\in M$, $r(x)\odot r(0)=r(0)\odot r(x)$.
\end{itemize}
For each $m\in\mathbb N$ and each $x\in M$, we define $r^0(x)=x$, $r^1(x)=r(x)$, and $r^{m+1}=r\circ r^{m}(x)=r(r^m(x))$, $m\ge 1$.
\end{defn}

We note that any weak square root $r:M\to M$ is a one-to-one map, indeed if $r(x)=r(y)$, then $x=r(x)\odot r(x)=r(y)\odot r(y)=y$. Moreover, if
$r_1$ and $r_2$ are two weak square roots on a pseudo MV-algebra $M$, then $r_1=r_2$. Indeed, by (Sq2), for each $x\in M$
$r_1(x)\odot r_1(x)\leq x$ implies $r_1(x)\leq r_2(x)$. In a similar way, $r_2(x)\leq r_1(x)$. That is $r_1=r_2$.
Now, we gather the main properties of square roots on pseudo MV-algebras that will be used in the sequel.
First, we note that if $r$ is a weak square root on a pseudo MV-algebra $M$, then $r(x)\odot r(y)=r(y)\odot r(x)$ implies that $x\odot y=y\odot x$. Indeed,
$x\odot y=(r(x)\odot r(x))\odot (r(y)\odot r(y))=(r(y)\odot r(y))\odot (r(x)\odot r(x))=y\odot x$.

\begin{prop}\label{3.2}
Let $r$ be a square root on a pseudo MV-algebra $(M;\oplus,^-,^\sim,0,1)$. Then for each $x,y\in M$, we have:
\begin{itemize}[nolistsep]
\item[{\rm (1)}] $x\leq x\vee r(0)\leq r(x)$, $r(1)=1$, $(r(x)\odot r(0))\vee(r(0)\odot r(x))\leq x$ and $r(x)\odot x=x\odot r(x)$.
\item[{\rm (2)}] $x\leq y$ implies that $r(x)\leq r(y)$.
\item[{\rm (3)}] $x\wedge y\leq r(x)\odot r(y), r(y)\odot r(x)$ and if $a\in \B(M)$ such that $a\leq r(0)$, then $a=0$.
\item[{\rm (4)}] $x\leq r(x\odot x)$ and $r(x\odot x)\odot r(x\odot x)=r(x)\odot r(x)\odot r(x)\odot r(x)=x\odot x$.
\item[{\rm (5)}] $(x\wedge x^-)\vee (x\wedge x^\sim)\leq r(0)$.
\item[{\rm (6)}] $r(x)\in \B(M)$ \iff $r(x)=x$.
\item[{\rm (7)}] $r(x)\wedge r(y)=r(x\wedge y)$.
\item[{\rm (8)}] $r(x)\ra r(y)\le r(x\ra y)$ and $r(x)\rightsquigarrow r(y)\le r(x\rightsquigarrow y)$. Moreover, $r(x)\odot r(y)\leq r(x\odot y)$ for all $x,y\in M$ if and only if $r(x)\ra r(y)= r(x\ra y)$ and $r(x)\rightsquigarrow r(y)= r(x\rightsquigarrow y)$.

\item[{\rm (9)}] $r(x\vee y)=r(x)\vee r(y)$.

\item[{\rm (10)}]
$r(x\odot y)\leq (r(x)\odot r(y))\vee r(0)$ and $r(x\odot x)=(r(x)\odot r(x))\vee r(0)$. Consequently, if $r(0)\leq x$, then $r(x\odot x)=x$.

\item[{\rm (11)}] {\rm (a)} $x\in \B(M)$ \iff $r(x)=x\oplus r(0)$ \iff $r(x)=r(0)\oplus x$.

\noindent
{\rm (b)} $(r(0)\ra 0)\odot (r(0)\ra 0)=(r(0)\rightsquigarrow 0)\odot (r(0)\rightsquigarrow 0)\in \B(M)$.

\noindent
{\rm(c)} If $a,b \in M$, $a\le b$, then $r([a,b])=[r(a),r(b)]$. In particular, $r(M)=[r(0),1]$.

\item[{\rm (12)}] $y\leq (r(x)\odot r(y))\wedge (r(y)\odot r(x))$ implies that $y\leq x$.

\item[{\rm (13)}] For each idempotent element $a\in \B(M)$, the map $r_a:[0,a]\to [0,a]$ sending $x$ to $r(x)\odot a$ is a square root on the pseudo MV-algebra $[0,a]$.

\item[{\rm (14)}] $r(x\oplus y)\ge(r(x)\odot r(0)^-)\oplus r(y)$. Moreover, if $r(x)\odot r(y)\leq r(x\odot y)$ for all $x,y\in M$, then the equality holds.

\item[{\rm (15)}] For each $m,n\in\mathbb N$ with $n\leq m$, $(r^m(x))^{2^n}=r^{m-n}(x)$.
\end{itemize}
Properties {\rm (1)--(8)} hold also for each weak square root on $M$.
\end{prop}

\begin{proof}
(1) $x\leq r(x)$ follows from (Sq1). On the other hand by \cite[Prop 1.16]{georgescu},
$(x\vee r(0))\odot (x\vee r(0))=(x\odot x)\vee (r(0)\odot x)\vee (x\odot r(0))\vee (r(0)\odot r(0))=
(x\odot x)\vee (r(0)\odot x)\vee (x\odot r(0))\vee 0\leq x$, so (Sq2) implies that $x\vee r(0)\leq r(x)$.
For each $x\in M$, $r(x)\odot x=r(x)\odot r(x)\odot r(x)=x\odot r(x)$.
Also, $r(x)\odot r(0),r(0)\odot r(x)\leq r(x)\odot r(x)=x$.

(2) From (Sq2) and $r(x)\odot r(x)=x\leq y$, we have $r(x)\leq r(y)$.

(3) By (2), $x\wedge y=r(x\wedge y)\odot r(x\wedge y)\leq r(x)\odot r(y)$.
Now, let $a\in \B(M)$ be such that $a\leq r(0)$. Then $a=a\odot a\leq r(0)\odot r(0)=0$.

(4) $x\odot x\leq x\odot x$ and (Sq2) imply that $x\leq r(x\odot x)$. The second part follows from
(Sq1): Indeed,
$r(x\odot x)\odot r(x\odot x)=x\odot x=r(x)\odot r(x)\odot r(x)\odot r(x)=x\odot x$.

(5)  $(x\wedge x^-)\odot (x\wedge x^-)\leq x\odot x^-=0$. Similarly,
$(x\wedge x^\sim)\odot (x\wedge x^\sim)\leq x^\sim\odot x=0$. Thus by (Sq2),
$x\wedge x^\sim,x\wedge x^-\leq r(0)$.

(6) If $r(x)\in \B(M)$, then by (Sq1), $x=r(x)\odot r(x)=r(x)\in\B(M)$. Conversely, if $r(x)=x$, then
$x=r(x)\odot r(x)\leq r(x)=x$ and so $r(x)\in\B(M)$.

(7) From (2) and $x\wedge y\leq x,y$, we have that $r(x\wedge y)\leq r(x)\wedge r(y)$. Also,
$\big(r(x)\wedge r(y) \big)\odot \big(r(x)\wedge r(y) \big)=(r(x)\odot r(x))\wedge (r(x)\odot r(y))\wedge
(r(y)\odot r(x))\wedge (r(y)\odot r(y))\leq x\wedge y$, so by (Sq2),
$r(x)\wedge r(y)\leq r(x\wedge y)$.

(8) Let $x,y\in M$. Then
\begin{eqnarray*}
(x\wedge y)\odot (r(x)\ra r(y))\odot (r(x)\ra r(y))&=& r(x\wedge y)\odot r(x\wedge y)\odot (r(x)\ra r(y))\odot (r(x)\ra r(y))\\
&\leq& r(x\wedge y)\odot r(x)\odot (r(x)\ra r(y))\odot (r(x)\ra r(y))\\
&=& r(x\wedge y)\odot (r(x)\wedge r(y))\odot (r(x)\ra r(y))\quad \text{by Prop \ref{PMV-prop}(i)}\\
&\leq& r(x\wedge y)\odot r(x)\odot (r(x)\ra r(y))\\
&\leq& r(x\wedge y)\odot r(y) \quad \text{by Prop \ref{PMV-prop}(i)}\\
&\leq& r(y)\odot r(y)=y.
\end{eqnarray*}
From Proposition \ref{PMV-prop}(vi), it follows that $(r(x)\ra r(y))\odot (r(x)\ra r(y))\leq (x\wedge y)\ra y=x\ra y$ and so by
(Sq1), $r(x)\ra r(y)\leq r(x\ra y)$.

Now, let $r(x)\odot r(y)\leq r(x\odot y)$ for all $x,y\in M$. Then
$r(x)\odot r(x\ra y)\leq r(x\odot (x\ra y))\leq r(y)$ which entails $r(x\ra y)\leq r(x)\ra r(y)$.
Therefore, $r(x\ra y)= r(x)\ra r(y)$.

Conversely, let $r(x\ra y)= r(x)\ra r(y)$ for all $x,y\in M$. Then from \cite[Prop 1.13]{georgescu} it follows that
$r(x)\ra r(x\odot y)=r(x\ra (x\odot y))=r(x^-\vee y)\geq r(y)$
and so $r(x)\odot r(y)\leq r(x\odot y)$.

Similarly, we can prove the relations with $\rightsquigarrow$.

(9) By part (7) and (Sq3), we have
\begin{eqnarray*}
r(x\vee y)&=& r\big(((x\to 0)\wedge (y\to 0))\rightsquigarrow 0\big)=\big(r(x\to 0)\wedge r(y\to 0)\big)\rightsquigarrow r(0)\\
&=& \big(r(x\to 0)\rightsquigarrow r(0) \big)\vee \big(r(y\to 0)\rightsquigarrow r(0)\big) \quad\text{by Prop \ref{PMV-prop}(x)}\\
&=& r((x\to 0)\rightsquigarrow 0)\vee r((y\to 0)\rightsquigarrow 0)=r(x)\vee r(y) \quad\text{by Prop \ref{PMV-prop}(ii)}.
\end{eqnarray*}

(10) By parts (7) and (8),
\begin{eqnarray*}
r(x\odot y)&=& r(((x\odot x)^-)^\sim)= r(((x\odot y)\ra 0)\rightsquigarrow 0)\\
&=& r\big((y\ra (x\ra 0))\rightsquigarrow 0\big)=r(y\ra (x\ra 0))\rightsquigarrow r(0)\quad \mbox{ by (Sq3)}\\
&\leq&\big(r(y)\ra r(x\ra 0)\big)\rightsquigarrow r(0) \mbox{ by (8) and Proposition \ref{PMV-prop}(xi)}\\
&=& \big(r(y)\ra (r(x)\ra r(0))\big)\rightsquigarrow r(0)\quad \mbox{ by (Sq3)}\\
&=& \big((r(x)\odot r(y))\ra r(0)\big)\rightsquigarrow r(0)\quad \mbox{ by (Sq3)}\\
&=& (r(x)\odot r(y))\vee r(0)\quad  \mbox{ by Proposition \ref{PMV-prop}(ii)}.
\end{eqnarray*}
It follows that $r(x\odot x)\leq (r(x)\odot r(x))\vee r(0)=x\vee r(0)$ for all $x\in M$. Also,
by parts (1) and (4), $x,r(0)\leq r(x\odot x)$, so $x\vee r(0)\leq  r(x\odot x)$.
Therefore, $r(x\odot x)=x\vee r(0)$ for all $x\in M$.

(11) 
First, we claim that
\begin{eqnarray}\label{y+y}
r(y\oplus y)=y\oplus r(0)=r(0)\oplus y,\quad \forall y\in M.
\end{eqnarray}
Check
\begin{eqnarray*}
r(y\oplus y)&=& r((y^\sim\odot y^\sim)^-)=r(y^\sim\odot y^\sim)\ra r(0) \mbox{ by (Sq3)}\\
&=& \big((r(y^\sim)\odot r(y^\sim))\vee r(0)\big)\ra r(0)=(y^\sim \vee r(0))\ra r(0) \mbox{ by (10)} \\
&=& (y^\sim \ra r(0))\wedge (r(0)\ra r(0))=y^\sim \ra r(0)=y\oplus r(0). \\
r(y\oplus y)&=& r((y^-\odot y^-)^\sim)=r(y^-\odot y^-)\rightsquigarrow r(0) \mbox{ by (Sq3)}\\
&=& \big((r(y^-)\odot r(y^-))\vee r(0)\big)\rightsquigarrow r(0)=(y^- \vee r(0))\rightsquigarrow r(0) \mbox{ by (10)} \\
&=& (y^- \rightsquigarrow r(0))\wedge (r(0)\rightsquigarrow r(0))=r(0)\oplus y.
\end{eqnarray*}

(a) We can show that $x\in \B(M)$ \iff $r(x)=x\oplus r(0)$ \iff $r(x)=r(0)\oplus x$. Indeed, if $x\in \B(M)$, then \eqref{y+y} yields $r(x)=x\oplus r(0)$. Conversely, let $r(x)=x\oplus r(0)$. Then (\ref{y+y}) implies $r(x\oplus x)=x\oplus r(0)=r(x)$ consequently, by (6), $x\oplus x=x$ and $x\in\B(M)$.

(b) Now, we show that $(r(0)\ra 0)\odot (r(0)\ra 0)\in \B(M)$. By the last paragraph, it suffices to prove that
$r((r(0)\ra 0)\odot (r(0)\ra 0))=((r(0)\ra 0)\odot (r(0)\ra 0))\oplus r(0)$. By (10),
$r((r(0)\ra 0)\odot (r(0)\ra 0))=(r(0)\ra 0)\vee r(0)$. On the other hand,
$((r(0)\ra 0)\odot (r(0)\ra 0))\oplus r(0)=(r(0)^-\odot r(0)^-)\oplus r(0)=r(0)^-\vee r(0)$ (see \cite[Prop 1.13(a)]{georgescu}).
It follows that $(r(0)\ra 0)\odot (r(0)\ra 0)\in \B(M)$. Using (\ref{y+y}) in a similar way, we can show that
$(r(0)\rightsquigarrow 0)\odot (r(0)\rightsquigarrow 0)\in \B(M)$.

(c)  Let $a,b\in M$ with $a\le b$  be given. By (2), $r([a,b])\s[r(a),r(b)]$. Let $r(a)\leq y\leq r(b)$.
Then $y=y\vee r(0)$ (since $y\geq r(a)\geq r(0)$) and by (10), $y=r(y\odot y)$. From
$a=r(a)\odot r(a)\leq y\odot y\leq r(b)\odot r(b)=b$ it follows that $y=r(y\odot y)\in r([a,b])$ consequently,
$r([a,b])=[r(a),r(b)]$.

As a consequence, we have $r(M)=r([0,1])=[r(0),r(1)]=[r(0),1]$.

(12) Let $y\leq r(y)\odot r(x),r(x)\odot r(y)$. Then
$r(y)\odot (r(x)\wedge r(y))=(r(y)\odot r(x))\wedge (r(y)\odot r(y))=(r(y)\odot r(x))\wedge y=y$. On the other hand,
$r(x)\wedge r(y)=r(y)\odot (r(y)^-\oplus r(x))$ which yields that
\begin{eqnarray*}
y&=& r(y)\odot (r(x)\wedge r(y))=r(y)\odot r(y)\odot (r(y)^-\oplus r(x))=y\odot (r(y)^-\oplus r(x))\\
&\leq & r(x)\odot r(y)\odot (r(y)^-\oplus r(x))=r(x)\odot (r(x)\wedge r(y))\leq r(x)\odot r(x)=x.
\end{eqnarray*}

(13)  Due to \cite[Prop 4.3]{georgescu}, we have $r_a(x)=r(x)\wedge a=a\odot r(x)$, $x\in [0,a]$.
Conditions (Sq1) and (Sq2) are clear. Since the mapping $\phi:M \to [0,a]$ defined by $\phi_a(x)=x\wedge a$, $x\in M$, is a homomorphism of pseudo MV-algebras, (Sq3) holds.

In addition, if $r$ is standard, then $r_a(x)\odot r_a(0)=(r(x)\odot a)\odot (r(0)\odot a)=r(x)\odot r(0)\odot a=r(0)\odot r(x)\odot a=
r(0)\odot a\odot r(x)\odot a=r_a(0)\odot r_a(x)$, which means $r_a$ is standard, too.

(14) By (8), $r(x\oplus y)=r((x^\sim)^-\oplus y)=r(x^\sim \ra y)\ge r(x^\sim)\ra r(y)=(r(x)\rightsquigarrow r(0))\ra r(y)=
(r(0)\oplus r(x)^\sim)^-\oplus r(y)=(r(x)\odot r(0)^-)\oplus r(y)$.

The second part also follows from (8).

(15) Choose $m\in\mathbb N$. If $n=1$, then $(r^m(x))^{2^n}=(r^m(x))\odot (r^m(x))=r(r^{m-1}(x))\odot r(r^{m-1}(x))=r^{m-1}(x)$. Now, let
$1\leq n$, $n+1\leq m$ and $(r^m(x))^{2^n}=r^{m-n}(x)$. Then
$(r^m(x))^{2^{n+1}}=(r^m(x))^{2^{n}}\odot (r^m(x))^{2^{n}}=r^{m-n}(x)\odot r^{m-n}(x)=r^{m-n-1}(x)$.
\end{proof}

\begin{prop}\label{pr:equi}
Let $r$ be a weak square root on a pseudo MV-algebra. The following statements are equivalent:
\begin{itemize}
\item[{\rm (i)}] $r$ is a square root on $M$.
\item[{\rm (ii)}] $r(x)\odot r(x^-)\leq r(0)$.
\item[{\rm (iii)}] $r(x^\sim)\odot r(x)\leq r(0)$.
\item[{\rm (iv)}]  $r(x^-)= r(x)\to r(0)$.
\item[{\rm (v)}] $r(x^\sim) = r(x)\rightsquigarrow r(0)$.
\end{itemize}
\end{prop}

\begin{proof}
By Proposition \ref{3.2}(8), we know that for a weak square root $r$, we always have $r(x)\to r(0)\leq r(x\to 0)=r(x^-)$ and $r(x)\rightsquigarrow r(0)\leq r(x\rightsquigarrow 0)=r(x^\sim)$.

Thus, let $r(x)\odot r(x^-)\leq r(0)$. It implies
 $r(x^-)\leq r(x)\to r(0)$.
Hence $r(x)\to r(0)=r(x^-)$.
Conversely, let $r(x)\to r(0)=r(x^-)$. Then $r(x^-)\leq r(x)\to r(0)$ which implies that $r(x)\odot r(x^-)\leq r(0)$.

Analogously, we can show that all statements (i)--(v) are equivalent.
\end{proof}

In the note just before Proposition \ref{3.2}, we have seen that if $r(x)\odot r(y)=r(y)\odot r(x)$, then $x\odot y=y\odot x$. The converse holds for $x,y\in M$ with
$r(0)\leq r(x)\odot r(y),r(y)\odot r(x)$.
The proof follows from part (10) of the latter proposition.

In the following proposition, we present some upper and lower bounds for $r(x)$, where $r$ is a square root on a pseudo MV-algebra $M$.

\begin{prop}\label{rmk-cor}
Let $r$ be a square root on a pseudo MV-algebra $(M;\oplus,^-,^\sim,0,1)$. Then:

\begin{itemize}
\item[{\rm (i)}]
\begin{eqnarray}\label{eq1-rmk-cor}
x\leq x\vee r(0)&=&r(x\odot x)\leq r(x),\quad \forall x\in M.\\
\label{eq1-rmk-cor3}
r(x^-)^\sim \vee r(x^\sim)^-&\leq& r(x^-)^\sim \oplus r(x^-)^\sim=x\leq x\vee r(0)=r(x\odot x)\leq r(x),\quad \forall x\in M.\\
\label{eq-rmk-cor'}
r(x^-)^\sim \oplus r(x^-)^\sim &=& x = r(x^\sim)^- \oplus r(x^\sim)^-,\quad \forall x\in M.
\end{eqnarray}
Moreover, $x\in\B(M)$ \iff $r(x)=x\vee r(0)$.

\item[{\rm (ii)}]
\begin{eqnarray}\label{eq:bound}
r(x)&\leq& (x\oplus r(0))\wedge (r(0)\oplus x),\quad \forall x\in M.
\end{eqnarray}
\item[{\rm (iii)}] For each $x\in M$, there exists a sequence
$\{x_n\}_{n\in \mathbb N}$ of elements of $M$ such that
$x\leq x_1\leq \cdots\leq x_n\le\cdots $ and
$x_n^{2n}=x$ for each $n\ge 1$.

\item[{\rm (iv)}]
\begin{equation}\label{eq:new}
r(x)\leq (r(0)\ra x)\wedge (r(0)\rightsquigarrow x),\quad \forall x\in M.
\end{equation}

\item[{\rm (v)}] \begin{eqnarray}\label{eq1-rmk-cor4}
r(x^\sim)^-\oplus r(0)=r(x)=r(0)\oplus r(x^-)^\sim,\quad \forall x\in M.
\end{eqnarray}
\item[{\rm (vi)}] \begin{eqnarray}\label{eq1-rmk-cor5}
x\oplus r(x)&=&r(x)\oplus x ,\quad \forall x\in M.\\
\label{eq1-rmk-cor6}
x\oplus r(0)&=&r(0)\oplus x,\quad \forall x\in M.
\end{eqnarray}
\item[{\rm (vii)}]
$r(0)=\max\{x\wedge x^-\colon x\in M\}=\max\{x\wedge x^\sim\colon x\in M\}$, $\forall x\in M$, and\\
$r(0)^-=\min\{x\vee x^-\colon x\in M\}=\min\{x\vee x^\sim\colon x\in M\}=r(0)^\sim$, $\forall x\in M$.

\item[{\rm (viii)}] $r(r(0)\ra 0)\leq (r(0)\ra 0)\vee r(r(0))$, $\forall x\in M$.
\end{itemize}
\end{prop}

\begin{proof}
{\rm (i)} In
{\rm Proposition \ref{3.2} (1)} and {\rm (4)}, it was proved that $x\leq x\vee r(0)\leq r(x)$ and $x\leq r(x\odot x)$.
Also, by part (10) from this proposition, $r(x\odot x)= (r(x)\odot r(x))\vee r(0)=x\vee r(0)$, so we have equation \eqref{eq1-rmk-cor}.
Hence, if $x\in M\setminus \B(M)$, then $x\odot x\neq x$ thus $r(0)\vee x<r(x)$. Otherwise, if $r(0)\vee x=r(x)$, then $r(x)=r(x\odot x)=x\vee r(0)$, that is $x\odot x=x$.
Moreover, $x\in \B(M)$ \iff $r(x)=r(0)\vee x$ (since $r$ is a one-to-one map).

By \eqref{eq1-rmk-cor}, replacing $x$ with $x^-$ and $x^\sim$, we have
$x^-\leq r(x^-)$ and $x^\sim\leq r(x^\sim)$ which imply that $r(x^-)^\sim \vee r(x^\sim)^-\leq x$.
Also, $r(x^-)^\sim\oplus r(x^-)^\sim=(r(x^-)\odot r(x^-))^\sim=x^{-\sim}=x$. In a similar way,
$x=r(x^\sim)^-\oplus r(x^\sim)^-$. Therefore, we have
\eqref{eq1-rmk-cor3}.

{\rm (ii)} According to the definition of a square root, $r(x)=\max\{z\in M\mid z\odot z\leq x\}$.
We claim that $x\oplus r(0)$ is an upper bound for $r(x)$.
Indeed, set $S:=\{z\in M\mid z\odot z\leq 0\}$. If $z\in M$ such that $z\odot z\leq x$, then $(z\odot z)\odot x^-=0$ and $x^\sim \odot (z\odot z)=0$ which means
$(z\odot x^-)\odot (z\odot x^-)\leq z\odot (z\odot x^-)=0$ and
$(x^\sim \odot z)\odot (x^\sim \odot z)\leq (x^\sim \odot z)\odot z=0$, thus
$x^\sim \odot z,z\odot x^-\in S$. It follows that $x^\sim \odot z,z\odot x^-\leq r(0)$ and hence,
\begin{eqnarray*}
z&\leq& x\vee z=x\oplus (x^\sim \odot z)\leq x\oplus r(0),\\
z &\leq& x\vee z=(z\odot x^-)\oplus x\leq r(0)\oplus x.
\end{eqnarray*}
Since $r(x)$ is the greatest element of $S$, we conclude that
\eqref{eq:bound} holds.

{\rm (iii)} Let $x\in M$ be given. It suffices to set $x_n:=r^n(x)$. Then by Proposition \ref{3.2}(15),
we have $x_n^{2n}=(r^n(x))^{2n}=x$ and $x_n=r^n(x)\leq r^{n+1}(x)=x_{n+1}$.

{\rm (iv)} We have $r(x)\odot r(0),r(0)\odot r(x)\leq r(x)\odot r(x)=x$ (by Proposition \ref{3.2}(2)), consequently,
by Proposition \ref{PMV-prop}(vi), $r(x)\leq r(0)\ra x$ and $r(x)\leq r(0)\rightsquigarrow x$.
That is, $r(x)\leq (r(0)\ra x)\wedge (r(0)\rightsquigarrow x)$.

{\rm (v)} Due to
\eqref{eq1-rmk-cor3} and \eqref{eq:bound}, $r(x^-)^\sim \leq x$ and $r(x)\leq r(0)\oplus x$. Then
\begin{eqnarray*}
r(0)\oplus r(x^-)^\sim &=& r(0)\oplus (r(x)\ra r(0))^\sim \quad \mbox{ by (Sq3)}\\
&=& r(0)\oplus (r(x)^-\oplus r(0))^\sim\\
&=&r(0)\oplus (r(0)^\sim\odot r(x))\\
&=& r(0)\vee r(x)=r(x) \quad \mbox{ by Proposition \ref{3.2}(2)}.
\end{eqnarray*}
In a similar way, $r(x)=r(x^\sim)^-\oplus r(0)$. Hence,
\eqref{eq1-rmk-cor4} is established.

(vi) (a) Let $x\in M$. By \eqref{eq-rmk-cor'}, we know that $r(x^-)^\sim\oplus r(x^-)^\sim=x=r(x^\sim)^-\oplus r(x^\sim)^-$.
On the other hand, by (\ref{eq1-rmk-cor4}), $r(x^\sim)^-\oplus r(0)=r(x)=r(0)\oplus r(x^-)^\sim$. It follows that
\begin{eqnarray*}
x\oplus r(x)&=& r(x^\sim)^-\oplus r(x^\sim)^-\oplus r(0)\oplus r(x^-)^\sim
= r(x^\sim)^-\oplus \left(r(x^\sim)^-\oplus r(0)\right)\oplus r(x^-)^\sim \\
&=& r(x^\sim)^-\oplus \left(r(0)\oplus r(x^-)^\sim\right)\oplus r(x^-)^\sim
= \left(r(x^\sim)^-\oplus r(0)\right)\oplus \left(r(x^-)^\sim\oplus r(x^-)^\sim\right) \\
&=& r(x)\oplus x.
\end{eqnarray*}
(b) For each $x\in M$, part (i) and equation (\ref{eq1-rmk-cor4}) imply that
\begin{eqnarray*}
x\oplus r(0)&=& r(x^\sim)^-\oplus r(x^\sim)^-\oplus r(0)=r(x^\sim)^-\oplus r(0)\oplus r(x^-)^\sim\\
&=& r(0)\oplus r(x^-)^\sim\oplus r(x^-)^\sim=r(0)\oplus x.
\end{eqnarray*}
Using \eqref{y+y}, we have (b).

(vii) By Proposition \ref{3.2}(5), for each $x\in M$, $x\wedge x^-\leq r(0)$. So, $r(0)$ is an upper bound for the set $\{x\wedge x^-\colon x\in M\}$ in $M$. Since $r(0)\odot r(0)=0$, we have $r(0)=r(0)\wedge r(0)^-\in \{x\wedge x^-\colon x\in M\}$, and
$r(0)=\max\{x\wedge x^-\colon x\in M\}$. In a similar way, we can prove the second part of the first equation.

The second equation follows from the negation of the first equation.

(viii) By Proposition \ref{3.2}(8) and Proposition \ref{PMV-prop}(iv), we have
\begin{eqnarray*}
(r(0)\ra 0)\vee r(r(0))&\geq& (r(r(0))\ra 0)\vee  r(r(0)) \quad \mbox{ since $r(0)\leq r(r(0))$}\\
&=& r(r(0))\ra \big(r(r(0))\odot r(r(0)) \big)\quad \mbox{ by Proposition \ref{PMV-prop}(iv)}\\
&=& r(r(0))\ra r(0)=r(r(0)\ra 0)\quad \mbox{ by (8)}.
\end{eqnarray*}

\end{proof}

In what follows, we need the following notions to present square roots on non-commutative examples of pseudo MV-algebras.

We say that an $\ell$-group $G$ is (i) {\it divisible} if for every integer $n\ge 1$ and $g\in G$, there is $h\in G$ such that $nh=g$, (ii) {\it two-divisible} if for every $g\in G$, there is $h\in H$ such that $h+h=g$. For example, the group of dyadic numbers $\mathbb D=\{m/2^n\mid m\in \mathbb Z,\, n\ge 1\}$ is two-divisible but not divisible. In the same way, we introduce a divisible pseudo MV-algebra and a two-divisible pseudo MV-algebra when we use the partial addition $+$ deduced from $\oplus$. We note that if $(G,u)$ is a unital $\ell$-group and $G$ is divisible or two-divisible, then so is $M=\Gamma(G,u)$.

An $\ell$-group $G$ enjoys {\it unique extraction of roots} if for every integer $n\ge 1$ and  $g,h\in G$, $ng=nh$ implies $g=h$. In the same way, we say that a pseudo MV-algebra enjoys unique extraction of roots. We note that every linearly ordered group, \cite[Lem 2.1.4]{Gla}, (linearly ordered pseudo MV-algebra) enjoys unique extraction of roots. The same is true for each representable $\ell$-group. Therefore, if $G$ or $M$ is two-divisible and it enjoys unique extraction of roots, then for each $x\in G$ $(x\in M$), there is a unique $y\in G$ ($y\in M$) such that $2y=x$, and we denote the unique element $y$ by $y:=x/2$.

Whence, if $G$ is a linearly ordered two-divisible group, then
(i) $x\le y$ implies $x/2\le y/2$ (otherwise $y/2<x/2$ giving a contradiction $y<x$), (ii) $(x\vee y)/2 = x/2\vee y/2$, and (iii) $(x\wedge y)/2 = (x/2)\wedge (y/2)$. Analogously, if a linearly ordered group $G$ is divisible, then
(i) $x\le y$ implies $x/n \le y/n$ for each integer $n\ge 1$, (ii) $(x\vee y)/n = (x/n)\vee (y/n)$, and (iii) $(x\wedge y)/n = (x/n) \wedge (y/n)$. The same holds if $G$ is a divisible representable $\ell$-group.

We note that according to Mal'cev \cite[Thm 7.3.2]{KoMe}, every locally nilpotent linearly ordered group can be embedded into a divisible locally nilpotent linearly ordered group, and every nilpotent linearly ordered group can be embedded into a divisible
linearly ordered group.

\begin{exm}\label{ex:sym}
Let $(H,1)$ be a subgroup of $(\mathbb R,1)$ that is two-divisible (e.g., $H=\mathbb Q$, rational numbers in $\mathbb R$, and $H=\mathbb D$, dyadic numbers in $\mathbb R$), and let $G$ be a (two)-divisible nilpotent linearly ordered group.

We note there are countably many mutually different subalgebras $(H,1)$ of $(\mathbb R,1)$ that are two-divisible. Indeed, let $p\ge 1$ be an integer. Define $H(p)=\{i/p^n\mid i\in \mathbb Z,\, n\ge 1\}$. Then $1\in H(p)$. If $p\ge 3$ is an even number, then $(H(p),1)$ is a two-divisible unital subgroup of $(\mathbb R,1)$. Let $p_0=2<p_1=3<\cdots<p_n$ be the first $n+1$ prime numbers and let $P_n=p_0p_1\cdots p_n$. Then $H(P_n)$ is two-divisible and for each $n\ge 1$, $1/p_n\in H(P_n)\setminus H(P_{n-1})$, and $1/P^k_n= p^k_{n+1}/P^k_{n+1}\in H(P_{n+1})$ for each $k$, so that $H(P_n)\subsetneqq H(P_{n+1})$ for each $n\ge 0$, and $\{(H(P_n),1)\}_{n=1}^\infty$ is a sequence of mutually different two-divisible unital subgroups of $(\mathbb R,1)$.

We note e.g. that $\mathbb D=H(2)=H(2^n)$ for each $n\ge 1$, and $\mathbb D\subseteq H$ for each two-divisible unital subgroup $(H,1)$ of $(\mathbb R,1)$.

(1) Define $M=\Gamma(H\lex G,(1,0))$. Then $M$ is a symmetric linearly ordered pseudo MV-algebra that is two-divisible, and it enjoys unique extraction of roots, but $M$ is not an MV-algebra. Moreover, if $u=(1,0)$, then $u/2=(1/2,0)\in \C(H\lex G)$.

(2) Define $M=\Gamma(H\lex G,(1,g_0))$, where $g_0\notin \C(G)$. Then $M$ is a linearly ordered pseudo MV-algebra that is two-divisible, and it enjoys unique extraction of roots, but $M$ is not symmetric.
\end{exm}

\begin{exm}\label{3.4}
(i) If $M$ is a Boolean algebra, then the identity map $\id_M:M\ra M$ is a square root.
Indeed, for each $x\in M$ we have $\id_M(x)\odot \id_M(x)=x\odot x=x$. Similarly, (Sq2) holds, and $\id_M$ satisfies (Sq3). Also, $\id_M$ is the only square root on $M$.

(ii) If in a pseudo MV-algebra  $(M;\oplus,^-,^\sim,0,1)$, the binary operation $\oplus$ is commutative, then $M$ is an MV-algebra and the concept of a square root on $M$ in Definition \ref{3.1} coincides with the concept of a square root on $M$ defined in \cite{Hol}.

(iii) Take a linearly ordered (or representable) two-divisible unital group $(G,u)$, where $u/2\in \C(G)$. According to Example \ref{ex:sym}(i), such groups exist. If we set $M=\Gamma(G,u)$, $M$ is
a symmetric two-divisible linearly ordered pseudo MV-algebra $M$.
First, we note that given $x\in G$, $x+u= x/2+x/2+u/2+u/2= x/2+u/2 +x/2+u/2
$ and $(x+u)/2=x/2+y/2$.
We claim $s:M\to M$ defined by
$$
s(x)=\frac{x+u}{2}, \quad x \in M,
$$
where $+$ is the group addition in the group $G$, is a square root on $M$. Since $0\le s(x)\le u$, $s(x)\in M$. Moreover, $s(x)=(x/2)+(u/2)$. (Sq1): It is clear that $s(x)\odot s(x) = (2((x+u)/2) -u)\vee 0=x$.  (Sq2): Let $y\odot y \le x$. Then $(y-u+y)\vee 0 = (2y-u)\vee 0\le x$ which gives $2y\le x+u= 2((x+u)/2)$, and $y\le (x+u)/2$. (Sq3): Let $x \in M$. Then $s(0)=u/2$ and $s(x^-)=(x^-+u)/2=((u-x)+u)/2=u-x/2$ and
$s(x)\ra s(0)=s(x)^-\oplus s(0)=(u-(x+u)/2+u/2)\wedge u=(u-x/2)\wedge u=u-x/2$. Moreover, $s$ is strict: $s(x)\odot s(0)=x/2=s(0)\odot s(x)$.

The example also works if $(G,u)$ is a two-divisible unital $\ell$-group which enjoys unique extraction of roots, $u/2\in \C(G)$, and $2y\le x$ entails $y\le x/2$.

(iv) Let $M$ be a direct product of symmetric two-divisible pseudo MV-algebras $M_i$ with square root for each $i\in I$. Then $M$ has a square root, namely if $x=(x_i)_i\in M=\prod_i M_i$, then $s(x)=(s_i(x_i))_i$, where $s_i$ is a unique square root on $M_i$, see part (iii).

(v) Let $M=\Gamma(G,u)$, where $(G,u)$ is an arbitrary linearly ordered (or representable) two-divisible unital $\ell$-group. Then
\begin{equation}\label{eq:form2}
r(x)=\frac{x-u}{2} +u,\quad x\in M,
\end{equation}
is a weak square root on $M$. Indeed, let $y\odot y\le x$. Then $(y-u+y)\vee 0\le x$ and $y-u+y\le x$. But $y-u+y=(y-u)+(y-u)+u\le x$, which gives $y\le (x-u)/2+u=r(x)$. It is clear that $r(x)\odot r(x)=x$. Moreover, $r(0)=u/2= (r(0))^-=(r(0))^\sim$. This $r$ is not necessarily standard, see Example \ref{example} and Example \ref{example1}.

In addition, if the strong unit $u$ is such that for each $g\in G$, $(g/2)+(u/2)=(g+u)/2$ (for example if $u/2 \in \C(G)$), then $r(x)=(x+u)/2$, $x\in M$, and $r$ is a square root: Check $r(x^-)=r(u-x)= (u-x-u)/2 +u = (u-x-u)/2 + u/2 +u/2= (u-x+u)/2$. On the other hand,
\begin{eqnarray*}
r(x)\to r(0)&=&(r(x))^-\oplus u/2= \big(u-((x-u)/2+u)\big) \oplus u/2\\
&=& ((u-x+u)/2)\wedge u= (u-x+u)/2=r(x^-),
\end{eqnarray*}
which by Proposition \ref{pr:equi} means $r$ is a square root on $M$.

(vi) Examples of weak square roots on non-symmetric pseudo MV-algebras that are not square roots will be done in Example \ref{example} and Example \ref{example1}.

(vii) If $G$ is an arbitrary $\ell$-group, then $M=\Gamma(\mathbb Z\lex G,(1,0))$ has no weak square roots.
\end{exm}

\begin{thm}\label{3.5}
Let $r$ be a square root on a pseudo MV-algebra $(M;\vee,\wedge,\oplus,0)$. Then $M$ is a Boolean algebras \iff $r(0)=0$.
\end{thm}

\begin{proof}
Let $r(0)=0$ and $x\in M$. By Proposition \ref{rmk-cor}\eqref{eq1-rmk-cor3}, we have $x\leq r(x\odot x)=r(0)\vee x\leq r(x)\leq r(0)\oplus x=x$ which implies that $r(x)=r(x\odot x)$ and so
$x=x\odot x$ (since $r$ is one-to-one). That is, $M=\B(M)$ is a Boolean algebra.

Conversely, let $M$ be a Boolean algebra. Then $r(0)\odot r(0)=r(0)$ and $r(0)\odot r(0)=0$. Therefore, $r(0)=0$.
\end{proof}

\begin{prop}\label{Homo}
Let $f:M_1\to M_2$ be a homomorphism of pseudo MV-algebras and $r$ be a square root on $M_1$. Then
$\tau:\im(f)\to \im(f)$, defined by $\tau(f(x))=f(r(x))$ for all $x\in M_1$, is a square root on $\im(f)$.
\end{prop}

\begin{proof}
(i) For each $x\in M_1$, we have $f(r(x))\odot f(r(x))=f(r(x)\odot r(x))=f(x)$.

(ii) Let $x,y\in M_1$ be such that $f(y)\odot f(y)\leq f(x)$. By Proposition \ref{3.2}(9) and (10),
$r((y\odot y)\vee x)=r(y\odot y)\vee r(x)=(y\vee r(0))\vee r(x)$, whence
\begin{eqnarray*}
\tau(f(x))&=&\tau\big((f(y)\odot f(y))\vee f(x)\big) =\tau\big(f((y\odot y)\vee x)\big)\\
&=& f\big(r((y\odot y)\vee x)\big)=f\big(r(y\odot y)\vee r(x)\big)\quad \mbox{ by Proposition \ref{3.2}(9)}\\
&=& f\big((y\vee r(0))\vee r(x)\big)\quad \mbox{ by Proposition \ref{3.2}(10)}\\
&=& (f(y)\vee f(0))\vee f(r(x))\\
&=& (f(y)\vee f(0))\vee \tau(f(x)).
\end{eqnarray*}
That is, $f(y)\leq \tau(f(x))$. Hence (Sq2) holds.

(iii) We have $\tau(f(x))^-=f(r(x))^-= f(r(x)^-)=f(r(x)\to r(0))=f(r(x))\to f(r(0))=\tau(f(x))\to \tau(f(0))$. Similarly for the second equality, that is, (Sq3) holds.

(iv) We show $\tau:\im(f)\to \im(f)$ is well-defined. Indeed, if $f(x)=f(y)$, then $f(r(x))$ and $f(r(y))$ are square roots of $f(x)=f(y)$.
Parts (i) and (ii) imply that $f(r(x))\odot f(r(x))=f(x)=f(y)$ and so $f(r(x))\leq f(r(y))$. In a similar way, $f(r(y))\leq f(r(x))$.
Thus, $\tau:\im(f)\to \im(f)$ sending $f(x)$ to $f(r(x))$ is a correctly defined square root.
\end{proof}

\begin{cor}\label{corHom}
Let $I$ be a normal ideal of a pseudo MV-algebra $(M;\oplus,^-,^\sim,0,1)$ with a square root $r$. Then:
\begin{itemize}[nolistsep]
\item[{\rm (i)}] The mapping $r_I:M/I\to M/I$ defined by $r_I(x/I)=r(x)/I$ is a square root
on $M/I$.
\item[{\rm (ii)}] For each $x,y\in M$, if $x\odot y^-,y\odot x^-\in I$, then $r(x)\odot r(y)^-,r(y)\odot r(x)^-\in I$.
\end{itemize}
\end{cor}

\begin{proof}
(i) Consider the natural homomorphism $\pi_I:M\to M/I$ sending $x$ to $x/I$. Since $\pi_I$ is onto, by Proposition \ref{Homo},
$r_I:M/I\to M/I$, defined by $r_I(x/I)=\pi_I(r(x))=r(x)/I$, is a square root.

(ii) Let $x,y\in M$ such that $x\odot y^-,y\odot x^-\in I$. Then $x/I=y/I$ and so by (i), $r(x)/I=r(y)/I$ which means
$r(x)\odot r(y)^-,r(y)\odot r(x)^-\in I$.
\end{proof}

Let $(M;\oplus,^-,^\sim,0,1)$ be a pseudo MV-algebra with a square root $r$. An ideal $I$ of $M$ is called {\em $r$-invariant}
if $r(I)\s I$. For example, each prime ideal $P$ of the pseudo MV-algebra $M$ is $r$-invariant, since
for each $x\in P$, $r(x)\odot r(x)=x\in P$ which entails $r(x)\in P$.

\begin{prop}\label{invIdeal}
Let $(M;\oplus,^-,^\sim,0,1)$ be a pseudo MV-algebra with a square root $r$ and $I$ be a normal ideal of $M$.
Then $I$ is $r$-invariant \iff $I$ is a Boolean ideal of $M$ (that is $x\wedge x^\sim \in I$ for all $x\in M$).
\end{prop}

\begin{proof}
Assume that $I$ is an $r$-invariant normal ideal of $M$. By Proposition \ref{corHom}, $r_I:M/I\to M/I$, defined in the latter corollary, is a square root on $M/I$.
Since $r_I(0/I)=r(0)/I=0/I$, Theorem \ref{3.5} implies that $M/I$ is a Boolean algebra and so $I$ is a Boolean ideal of $M$.
Conversely, let $I$ be a Boolean ideal of $M$. Then $M/I$ is a Boolean algebra and $r_I:M/I\to M/I$ is the identity map on $M/I$.
Hence for each $x\in I$, we have $r(x)/I=r_I(x/I)=x/I=0/I$, that is $r(x)\in I$. Therefore, $I$ is $r$-invariant.
\end{proof}

\begin{rmk}\label{3.6}
Let $(M;\oplus,^-,^\sim,0,1)$ be a pseudo MV-algebra and $a\in \B(M)\setminus \{0,1\}$. Set $b:=a^-$. Then $[0,a]$ and $[0,b]$ are proper subsets of $M$.
Consider the pseudo MV-algebras $([0,a];\oplus,^{-a},^{\sim a},0,a)$ and $([0,b];\oplus,^{-b},^{\sim b},0,b)$
(see the note just before Proposition \ref{PMV-prop}).
The map $\varphi:M\to [0,a]\times [0,b]$ sending
$x\in M$ to $(x\wedge a,x\wedge b)$ is an isomorphism of pseudo MV-algebras (see \cite[Page 23]{georgescu}).
Hence, each pseudo MV-algebra with the non-empty set $\B(M)\setminus\{0,1\}$ is not directly indecomposable.
Now, let $r$ be a square root on $M$. By Proposition \ref{3.2}(13), $r_a$ and $r_b$ are square roots on the pseudo MV-algebras $[0,a]$ and
$[0,b]$, respectively. Therefore, in this case, $M$ has a decomposition of pseudo MV-algebras with square roots.
\end{rmk}

B{\v{e}}lohl{\'a}vek in \cite[Cor 3]{Be} showed that the class of all commutative residuated lattices with square roots is a variety. Consequently, the class of MV-algebras with square roots is also a variety.
The following theorem shows that the class of pseudo MV-algebras with square roots is a variety, too.

\begin{thm}\label{variety}
The class of all pseudo MV-algebras with square roots is a variety. The same is true for the class of pseudo MV-algebras with weak square roots.
\end{thm}

\begin{proof}
Let $V$ be the class of all pseudo MV-algebras $(M;\oplus,^-,^\sim,0,1)$ such that there is a square root on $M$.
Consider the class $W$ of algebras of type $(M;\oplus,^-,^\sim,r,0,1)$ (2,1,1,1,0,0) satisfying the identities (A1)--(A8) and the following additional three identities:
\begin{itemize}[nolistsep]
\item[(1)] $r(x)\odot r(x)=x$;
\item[(2)] $r((y\odot y)\vee x)\wedge y=y$;
\item[(3)] $r(x^-)=r(x)\to r(0)$ and $r(x^\sim)=r(x) \rightsquigarrow r(0)$.
\end{itemize}

First, assume that $M\in V$. There exists a square root $r$ on $M$. We show that $(M;\oplus,^-,^\sim,r,0,1)$ belongs to $W$.
We only need to show (2). For each $x,y\in M$ by Proposition \ref{3.2}(9) and Proposition \ref{rmk-cor}(i),
$r((y\odot y)\vee x)=r(y\odot y)\vee r(x)\geq y\vee r(x)$ and so $r((y\odot y)\vee x)\wedge y=y$. Hence
$(M;\oplus,^-,^\sim,r,0,1)\in W$.

Conversely, let $(M;\oplus,^-,^\sim,r,0,1)$ belong to $W$. Clearly by (A1)--(A8), $(M;\oplus,^-,^\sim,0,1)$ is a
pseudo MV-algebra. Let $y\odot y\leq x$. Then by (2), $y=r((y\odot y)\vee x)\wedge y=r(x)\wedge y$ which means $y\leq r(x)$.
From (3), we get (Sq3). Hence, $(M;\oplus,^-,^\sim,0,1)$ is a pseudo MV-algebra with the square root $r$. That is,
$(M;\oplus,^-,^\sim,r,0,1)\in V$.
Therefore, $V$ is a variety.

Using equations (1)--(2), we can prove that the class of pseudo MV-algebras with weak square roots is also a variety.
\end{proof}

Let $M_1$ and $M_2$ be pseudo MV-algebras and $r_1:M_1\to M_1$ and $r_2:M_2\to M_2$ be square roots.
If $f:M_1\to M_2$ is a homomorphism of pseudo MV-algebras, then
for each $x\in M$, $f(r(x))\odot f(r(x))=f(x)$ and so by (Sq2), $f(r_1(x))\leq r_2(f(x))$. We say the homomorphism $f$
{\em preserves} square roots if
$f(r_1(x))=r_2(f(x))$ for all $x\in M$. In the next theorem, we give a necessary and sufficient condition under which
a homomorphism of pseudo MV-algebras preserves square roots.

\begin{thm}\label{corHome}
{\rm (1)} Let $M_1$ and $M_2$ be pseudo MV-algebras and $r_1:M_1\to M_1$ and $r_2:M_2\to M_2$ be square roots.
If $f:M_1\to M_2$ is a homomorphism of pseudo MV-algebras, then $f$ preserves square roots \iff $\im(f)$ is closed under $r_2$.

{\rm (2)} If $f:M_1\to M_2$ is an isomorphism and $r_1$ is a (weak) square root on $M_1$, then $r_2=f\circ r_1\circ f^{-1}$ is a (weak) square root on $M_2$.
\end{thm}

\begin{proof}
(1) ($\Leftarrow$) Suppose that $\im(f)$ is closed under $r_2$. Then $r_2\big|_{\im(f)}:\im(f)\to \im(f)$ is a square root on
the pseudo MV-algebra $\im(f)$.
Since by Proposition \ref{Homo}, the map $\tau:\im(f)\to \im(f)$ defined by $\tau(f(x))=f(r_1(x))$, $x\in M_1$, is a square root,
then $\tau=r_2$ which implies that $r_2(f(x))=\tau(f(x))=f(r_1(x))$ for all $x\in M$.

($\Rightarrow$) Let $f$ preserve square roots. For each $f(x)\in\im(f)$, we have $r_2(f(x))=f(r_1(x))\in\im(f)$.
Therefore, $\im(f)$ is closed under $r_2$.

(2) (Sq1): Let $z\in M_2$. Then $r_2(z)\odot r_2(z)=f(r_1(f^{-1}(z)))\odot f(r_1(f^{-1}(z))) = f(r_1(f^{-1}(z))\odot r_1(f^{-1}(z)))=f(f^{-1}(z))=z$.

(Sq2): Let $y,x\in M_2$ and $y\odot y\le x$. Then $f^{-1}(y)\odot f^{-1}(y)\le f^{-1}(x)$, so that $f^{-1}(y)\le r_1(f^{-1}(x))$ and $y\le f(r_1(f^{-1}(x)))=r_2(x)$.

(Sq3): Let $x\in M_2$, so $r_2(0)=f(r_1(0))$, and $r_2(x^-)=f(r_1(f^{-1}(x)^-))= f(r_1(f^{-1}(x))\to f(r_1(0))=r_2(x)\to r(0)$.
\end{proof}

\begin{lem}\label{3.9}
Let $r$ be a square root on a pseudo MV-algebra $(M;\oplus,^-,^\sim,0,1)$. Then $r(0)^-=r(0)^\sim$, and there exists a unique idempotent element $u\in \B(M)$ such that
$u\vee r(0)=r(0)^-=r(0)^\sim$.
\end{lem}

\begin{proof}
Let $u:=r(0)^-\odot r(0)^-$. By Proposition \ref{3.2}(11), $u\in \B(M)$ and we have
\begin{eqnarray*}
u\vee r(0)&=& (r(u)\odot r(u))\vee r(0)=r(u\odot u) \quad\mbox{ by Proposition \ref{3.2}(10) }\\
&=& r(u)\quad \mbox{ since $u\in \B(M)$}\\
&=& r(r(0)^-\odot r(0)^-)\quad \mbox{ by the assumption}\\
&=& \big(r(r(0)^-)\odot r(r(0)^-)\big)\vee r(0)\quad \mbox{ by Proposition \ref{3.2}(10) }\\
&=& r(0)^-\vee r(0)=r(0)^-,
\end{eqnarray*}
since $r(0)\odot r(0)=0$ which yields $r(0)\le r(0)^-$.
In a similar way, using Proposition \ref{3.2}(11), we can show that $u\vee r(0)=r(0)^\sim$.

Now, let $v\in \B(M)$ such that $v\vee r(0)=r(0)\to 0$. Then we have
\begin{eqnarray*}
r(0)^-&=& v\vee r(0)=(r(v)\odot r(v))\vee r(0)=r(v\odot v),\mbox{ by Proposition \ref{3.2}(10)}\\
&=& r(v), \mbox{ since $v\in \B(M)$}.
\end{eqnarray*}
In a similar way, $r(u)=r(0)^-$. It follows that $u=r(u)\odot r(u)=r(v)\odot r(v)=v$.
\end{proof}

We finish this section with the following result.
In Proposition \ref{3.2}(11)(c), we showed that if $r$ is a square root on a pseudo MV-algebra $(M;\oplus,^-,^\sim,0,1)$, then $\im(r)=[r(0),1]$. In the sequel, we show that $\im(r)$ can be converted into a pseudo MV-algebra such that it is isomorphic to $M$.

\begin{thm}\label{intv1}
Let $(M;\oplus,^-,^\sim,0,1)$ be a pseudo MV-algebra with a square root $r$. Then the structure $([r(0),1];\boxplus,',^*,r(0),1)$ is a pseudo MV-algebra, where for given $r(x),r(y) \in r(M)$,
\begin{eqnarray*}
r(x)\boxplus r(y)= r(x\oplus y),\quad r(x)'= r(x)^-\oplus r(0),\quad r(x)^*=r(x)^\sim\oplus r(0).
\end{eqnarray*}

Moreover, the pseudo MV-algebras $(M;\oplus,^-,^\sim,0,1)$ and $([r(0),1];\boxplus,',^*,r(0),1)$ are isomorphic, and the restriction of $r$ onto $r(M)$ is a square root on $r(M)$.
\end{thm}

\begin{proof}
(i) First, we note that $r$ is one-to-one, so by Proposition \ref{3.2}(11)(c), the operations $\boxplus$, $'$ and $^*$ are well-defined, and $[r(0),1]$ is closed under these operations.
We verify conditions (A1)--(A8). Let $x,y,z\in M$ be given. \\
(1)
\begin{eqnarray*}
r(x)\boxplus (r(y)\boxplus r(z))&=&r(x)\boxplus r(y\oplus z)=r(x\oplus (y\oplus z))= r((x\oplus y)\oplus z)\\
&=&r(x\oplus y)\boxplus r(z)=(r(x)\boxplus r(y))\boxplus r(z).
\end{eqnarray*}
(2) $r(x)\boxplus r(0)=r(x\oplus 0)=r(x)=r(0\oplus x)=r(0)\boxplus r(x)$. \\
(3) $r(x)\boxplus r(1)=r(x\oplus 1)=r(1)=r(1\oplus x)=r(1)\boxplus r(x)$. \\
(4) $1'=1^-\oplus r(0)=0\oplus r(0)=r(0)$. Similarly, $1^*=r(0)$. \\
(5) First, we note that by (Sq3) and Proposition \ref{rmk-cor}(vi), for each $x\in M$, the following identities hold:
\begin{eqnarray}\label{intv2} r(x)^*&=& r(x)^\sim\oplus r(0)=r(0)\oplus r(x)^\sim=r(x)\rightsquigarrow r(0)=r(x\rightsquigarrow 0)=r(x^\sim),\\
\label{intv3} r(x)'&= & r(x)^-\oplus r(0)=r(x)\to r(0)=r(x\to 0)=r(x^-).
\end{eqnarray}
It follows that
$(r(x)')^*=r(x^-)^*=r(x^{-\sim})=r(x)$ and
$(r(x)^*)'=r(x^\sim)'=r(x^{\sim-})=r(x)$. \\
(6) From (\ref{intv2}) and (\ref{intv3}) deduce that
\begin{eqnarray*}
(r(x)'\boxplus r(y)')^*=(r(x^-)\boxplus r(y^-))^*=r(x^-\oplus y^-)^*=r((x^-\oplus y^-)^\sim)=r((x^\sim\oplus y^\sim)^-).
\end{eqnarray*}
A similar calculus shows that $(r(x)^*\boxplus r(y)^*)'=r((x^\sim\oplus y^\sim)^-)$. Hence $(r(x)'\boxplus r(y)')^*=(r(x)^*\boxplus r(y)^*)'$. \\
(7) Let $u\boxdot v:=(v'\boxplus u')^*$ for all $u,v\in [r(0),1]$. Then from (\ref{intv2}) and (\ref{intv3}) it follows that
\begin{eqnarray}
\label{intv4} r(x)\boxdot r(y)= (r(y)'\boxplus r(x)')^*=(r(y^-)\boxplus r(x^-))^*=r(y^-\oplus x^-)^*=r((y^-\oplus x^-)^\sim)=r(x\odot y).
\end{eqnarray}
Similarly to (6), using (\ref{intv2})--(\ref{intv4}), we can prove
(A6) and (A7) hold.
Therefore, $([r(0),1];\boxplus,',^*,r(0),1)$ is a pseudo MV-algebra.

(ii) Consider the mapping $f:M\to [r(0),1]$ defined by $f(x)=r(x)$. Clearly, $f(x)\in \im(r)=[r(0),1]$.
By (\ref{intv2}) and (\ref{intv3}), for each $x\in M$, $f(x)^*=r(x)^*=r(x^\sim)=f(x^\sim)$ and
$f(x)'=r(x)'=r(x^-)=f(x^-)$.
In addition, for each $x,y\in M$, $f(x)\boxplus f(y)=r(x)\boxplus r(y)=r(x\oplus y)=f(x\oplus y)$, and $f(0)=r(0)$, $f(1)=r(1)=1$.
Therefore, $f$ is an isomorphism, $M$ is isomorphic to $\im(r)$, and using Theorem \ref{corHome}(2), for each $x\in M$, we have $f(r(f^{-1}(f(x))) = f(r(x)) =r(r(x))$, is a square root on $r(M)=[r(0),1]$.
\end{proof}

\begin{rmk}\label{intv5}
Let $(M;\oplus,^-,^\sim,0,1)$ be a pseudo MV-algebra with a square root $r$.

(i) We have $M\cong [r(0),1]\cong [r^2(0),1]\cong\cdots\cong [r^n(0),1]$, $n\ge 1$, and the restriction of $r$ onto $[r^n(0),1]$ is a square root on $r^n(M)=[r^n(0),1]$.

(ii) (a) If $M$ is a Boolean algebra, then $r=\id_M$ and $M=r(M)$. (b) If $M$ is strict, then $[r(0),1]$ is strict, since
$r(r(0))=r(r(0)^-)=r(r(0))\ra r(0)=r(r(0))^-\oplus r(0)=r(r(0))'$. (c) If $r(0)>0$ and $r$ is not strict, by Theorem \ref{3.10}(iii), $M\cong M_1\times M_2$ such that $M_1$ is a Boolean algebra and $M_2$ is a strict pseudo MV-algebra and they both are uniquely determined by $M$. In a similar way, $[r(0),1]\cong M_{1,1}\times M_{2,1}$.
Theorem \ref{3.10} implies that $M_1\cong M_{1,1}$ and $M_2\cong M_{2,1}$.
We can continue this process $n$-times and from (i) we conclude that $[r^n(0),1]\cong M_{1,n}\times M_{2,n}$, where
$M_{1,n}$ is a Boolean algebra and $M_{2,n}$ is a strict pseudo MV-algebra. Whence
$M_1\cong M_{1,n}$ and $M_2\cong M_{2,n}$.
\end{rmk}

\section{Pseudo MV-algebras with Strict Square Roots}

In the section, we present an important class of pseudo MV-algebras with strict square roots. These pseudo MV-algebras form a subvariety of the variety of pseudo MV-algebras with square root. The complete characterization of strict pseudo MV-algebras will be postponed to the next section. In the end of the section, we introduce strongly atomless pseudo MV-algebras which imply the strictness of pseudo MV-algebras.

Consider a pseudo MV-algebra $M$ with a square root $r:M\to M$. From $r(0)\odot r(0)=0$, we get that

\begin{eqnarray}\label{eq:strict}
\label{r(0)} r(0)\leq r(0)^-\wedge r(0)^\sim= r(0)^-=r(0)^\sim.
\end{eqnarray}
In a particular case, if in \eqref{eq:strict} we have equality, the square root will say to be strict:

\begin{defn}\label{3.7}
A pseudo MV-algebra $(M;\oplus,^-,^\sim,0,1)$ with a square root (weak square root) $r:M\to M$ is called {\em strict} if
$r(0)=r(0)^-$.
\end{defn}

Due to Lemma \ref{3.9}, a square root $r$ on a pseudo MV-algebra $M$ is strict \iff $r(0)=r(0)^\sim$.

For example, (iii) and (v) of Example \ref{3.4} give a strict square root and a strict weak square root, respectively, on a pseudo MV-algebra, whereas \ref{3.4}(i) is an example of a pseudo MV-algebra whose square root is not strict. Another example of a strict weak square root on a non-symmetric pseudo MV-algebra will be done in Example \ref{example}.

Consider the assumptions of Theorem \ref{corHome}. If $r_1:M_1\to M_1$ is strict and $f:M_1\to M_2$ preserves square roots, then
$r_2(0)=f(r_1(0))=f(r_1(0)^-)=(f(r_1(0)))^-=(r_2(f(0)))^-=(r_2(0))^-$. So, $r_2$ is strict, too.

We introduce the following notions. Let $X$ be a non-empty subset of  a partially ordered set $(P,\le)$. We define $\uparrow X:=\{a\in P\mid x\le a$ for some $x\in P\}$ and $\downarrow X:=\{a\in P\mid a\le x$ for some $x\in P\}$.

\begin{prop}\label{strict-pro}
Let $s:M\to M$ be a strict square root on a pseudo MV-algebra $(M;\oplus,^-,^\sim,0,1)$. The following statements hold:
\begin{itemize}[nolistsep]
\item[{\rm (i)}] For each $x\in M$, $s(x)^-\vee s(x)^\sim\leq s(0)$ and $s(0)\oplus s(x)=s(x)\oplus s(0)=1$.

\item[{\rm (ii)}] For each $b\in \B(M)$, $s(0)\leq b\leq 1$ implies that $b=1$, consequently, $(\downarrow s(0)\cup \uparrow s(0))\cap \B(M)=\{0,1\}$ and if $M$ is not proper, then $\B(M)\cap s(M)=\{1\}$.

\item[{\rm (iii)}]  For all $x,y\in M$, $s(x)\to (s(x)\odot s(y))=s(y)=s(x)\rightsquigarrow (s(y)\odot s(x))$.

\item[{\rm (iv)}] The square root $s$ is standard. 

\item[{\rm (v)}] For all $x\in M$, there exists $z_x\in M$ such that $x=z_x\oplus z_x$,
 and $x\leq s(0)\oplus z_x$.

\item[{\rm (vi)}]  For all $x\in M$, we have $(s(0)\oplus z_x)^2=x$.
\end{itemize}
\end{prop}

\begin{proof}
(i) For each $x\in M$, we have $s(0)\leq s(x)$ and so $s(x)^-\leq s(0)^-=s(0)$ and $s(x)^\sim\leq s(0)^\sim =s(0)$. Due to \cite[Prop 1.9]{georgescu}, we have $x\oplus s(0)=1=s(0)\oplus x$.

(ii) Let $b\in \B(M)$ such that $s(0)\leq b\leq 1$. First we note that if $0\neq 1$, then $s(0)<b$
(otherwise, by Proposition \ref{3.2}(6), $0=s(0)\odot s(0)=b\odot b=b$). Then $b^-<s(0)^-=s(0)$ and so $b^-=(b^-)^2\leq s(0)\odot s(0)=0$ which implies that
$b=1$. Therefore, by Proposition \ref{3.2}(3),  $(\downarrow s(0)\cup \uparrow s(0))\cap \B(M)=\{0,1\}$.
Clearly, $1=s(1)\in \B(M)\cap s(M)$. Now, let $b\in \B(M)\cap s(M)$. Then by (i), $b^-\leq s(0)$ and so $b=0$ which means that $b=1$.

(iii)  By Proposition \ref{PMV-prop}(iv),
\begin{eqnarray*}
s(x)\to (s(x)\odot s(y))&=& s(x)^-\vee s(y),\\
s(x)\rightsquigarrow (s(y)\odot s(x))&=& s(x)^\sim\vee s(y).
\end{eqnarray*}
By (i), $s(x)^-\vee s(x)^\sim\leq s(0)\leq s(y)$, so we have $x,y\in M$ $s(x)\to (s(x)\odot s(y))=s(y)=s(x)\rightsquigarrow (s(y)\odot s(x))$.

(iv) First, we note that it is possible to show that (Sq4) is equivalent to $s(0)\odot s(x)=s(x)\odot s(0)$ for all $x\in M$. Indeed, (Sq4) implies directly
$s(0)\odot s(x)=s(x)\odot r(0)$. Conversely, we have $s(0)\odot x=s(0)\odot s(x)\odot s(x)=s(x)\odot s(0)\odot s(x)=s(x)\odot s(x)\odot s(0)=x\odot s(0)$.

Therefore, for each $x\in M$, by Proposition \ref{rmk-cor}\eqref{eq1-rmk-cor6}, we get
$s(0)\odot x=s(0)^-\odot x=(x^\sim\oplus s(0))^-=(s(0)\oplus x^\sim)^-=x\odot s(0)$, that is $s$ is strict.

(v) By (Sq1), we have that $x=s(x^-)^\sim\oplus s(x^-)^\sim$. Set $z_x:=s(x^-)^\sim$, then $x=z_x\oplus z_x$.
Property (iv) implies that $s(0)+ z_x= s(0)\oplus z_x=z_x\oplus s(0)= z_x+ s(0)$. Let $(G,u)$ be a unital $\ell$-group such that $M=\Gamma(G,u)$. Taking into account that $s(0)=s(0)^\sim=-s(0)+u$, so $-u+s(0)=-s(0)$, we have
\begin{eqnarray*}
-x+(z_x+s(0))&=& -(z_x\oplus z_x)+z_x+s(0)=-\big((z_x+z_x)\wedge u \big)+z_x+s(0)\\
&=& \big((-z_x-z_x)\vee -u \big)+z_x+s(0)\\
&=& \big(-z_x-z_x+z_x+s(0) \big)\vee \big(-u+z_x +s(0)\big)\\
&=& \big(-z_x+s(0) \big)\vee \big(-s(0)+z_x \big)\quad \text{ use the second part of (iv)}\\
&=&|-s(0)+z_x|\geq 0.
\end{eqnarray*}

(vi) Let $x\in M$ and $s$ be strict. By (iv), we have
\begin{eqnarray*}
(s(0)\oplus z_x)^2&=& (z_x\oplus s(0))\odot (s(0)\oplus z_x)=\Big(\big((z_x+s(0))\wedge u\big)-u+ \big((s(0)+z_x)\wedge u\big)\Big)\vee 0\\
&=& \Big(\big((z_x+s(0)-u)\wedge 0\big)+ \big((s(0)+z_x)\wedge u\big)\Big)\vee 0\\
&=& \Big(\big(z_x+s(0)-u+s(0)+z_x\big)\wedge u\wedge (z_x+s(0))\wedge (s(0)+z_x)\Big)\vee 0\\
&=& \Big(\big(z_x+s(0)-u+s(0)+z_x\big)\wedge u\wedge (s(0)\oplus z_x) \wedge (z_x\oplus s(0))\Big)\vee 0\\
&=& \big((z_x\oplus (s(0)-u+s(0))\oplus z_x\big)\wedge (s(0)\oplus z_x)\quad \mbox{ by the assumption.}
\end{eqnarray*}
Note that $s(0)+s(0)=(u-s(0))+s(0)=u$, so $s(0)=u-s(0)$ entails that $-u+s(0)=-s(0)$.
Hence, $s(0)-u+s(0)=0$, consequently we have
\begin{eqnarray*}
(s(0)\oplus z_x)^2&=& (z_x\oplus 0\oplus z_x)\wedge (s(0)\oplus z_x)\\
&=& x\wedge (s(0)\oplus z_x)=x \quad \mbox{ by (v)}.
\end{eqnarray*}
\end{proof}

From Proposition \ref{strict-pro}(iv) and Proposition \ref{rmk-cor}(1), it follows that if $s$ is a strict square root on a pseudo MV-algebra $M$, then, for each $x\in M$, there is $z_x\in M$ such that $x=z_x\oplus z_x$ and $z_x\oplus s(0)=s(0)\oplus z_x$.

\begin{thm}\label{3.10}
Let $r:M\to M$ be a square root on a pseudo MV-algebra $(M;\oplus,^-,^\sim,0,1)$. Then $M$ satisfies only one of the following statements:
\begin{itemize}[nolistsep]
\item[{\rm (i)}] The pseudo MV-algebra $M$ is a Boolean algebra.
\item[{\rm (ii)}] The pseudo MV-algebra $M$ is a strict pseudo MV-algebra.
\item[{\rm (iii)}] The pseudo MV-algebra $M$ is isomorphic to the direct product $M_1\times M_2$, where $M_1$ is a
Boolean algebra and $M_2$ is a strict pseudo MV-algebra. Moreover, the Boolean algebra $M_1$ and the strict pseudo MV-algebra $M_2$ are uniquely determined by $M$ up to isomorphism.
\end{itemize}
\end{thm}

\begin{proof}
Due to Proposition \ref{3.2}(11), the element $u:=r(0)^-\odot r(0)^-$ is a Boolean element. We have three cases.

Case 1. If $u=1$, then $r(0)^-=1$ entails that $r(0)=0$. Theorem \ref{3.5} implies $M$ is a Boolean algebra.

Case 2. If $u=0$, then $r(0)^-\odot r(0)^-=0$ and so $r(0)^-\leq r(0)$. In view of (\ref{r(0)}), $r(0)=r(0)^-$, that is, $M$ is a strict pseudo MV-algebra.

Case 3. Let $u\notin \{0,1\}$. Consider pseudo MV-algebras
$([0,u];\oplus,^{-u},^{\sim u},0,u)$ and $([0,v];\oplus,^{-v},^{\sim v},0,v)$, where $v=u^-=u^\sim$.
By Proposition \ref{3.2}(13), they have square roots $r_u:[0,u]\to [0,u]$ and $r_v:[0,v]\to [0,v]$ respectively;
recall that $r_u(x)=r(x)\odot u$ and $r_v(y)=r(y)\odot v$.
We have
\begin{eqnarray}\label{159}
r_u(0)&=& r(0)\odot u=r(0)\odot \big(r(0)^-\odot r(0)^-\big)=0. \\
\label{1591} r_v(0)&=& r(0)\odot v=r(0)\odot \big(r(0)^-\odot r(0)^-\big)^\sim=r(0)\odot (r(0)\oplus r(0))=r(0).
\end{eqnarray}
Note that $(r(0)\oplus r(0))=(r(0)^-\odot r(0)^-)^\sim$ is an element of $\B(M)$ and so
$r(0)\odot (r(0)\oplus r(0))=r(0)\wedge (r(0)\oplus r(0))=r(0)$.
From (\ref{159}) and Theorem \ref{3.5}, we get that $[0,a]$ is a Boolean algebra, and from (\ref{1591}) and Proposition \ref{3.9}, we have $[0,v]$ is a strict pseudo MV-algebra.

In addition, by Remark \ref{3.6}, $M\cong [0,u]\times [0,v]$.

Now, we prove the second part of (iii), which is the uniqueness of the direct product. Let $M\cong_{\varphi} M_1\times M_2$, where $M_1$ is a Boolean algebra
and $M_2$ is a strict pseudo MV-algebra.
Let $0_1$ and $0_2$ ($1_1$ and $1_2$) be the least (greatest) elements of $M_1$ and $M_2$, respectively.
Consider the square roots $s_1:M_1\to M_1$ and $s_2:M_1\to M_2$ (by the assumption, $s_1:=\id_{M_1}$ and $s_2$ is strict).
Clearly, $s:=(s_1,s_2):M_1\times M_2\to M_1\times M_2$, defined by $s(x_1,x_2)=(s_1(x_1),s_2(x_2))$, is a square root on $M_1\times M_2$. Also,
\begin{eqnarray*}
s(0_1,0_2)\to (0_1,0_2)&=& (s_1(0_1),s_2(0_2))\to (0_1,0_2)=(0_1\to 0_1,s_2(0_2)\to 0_2)\\
&=& (1_1,s_2(0_2)\to 0_2)=(1_1,0_2)\vee (0_1,s_2(0_2)\to 0_2)\\
&=& (1_1,0_2)\vee (0_1,((s_2(0)\to 0)\odot (s_2(0)\to 0))\vee s_2(0))\quad \mbox{ by Lemma \ref{3.9}}\\
&=& (1_1,0_2)\vee (0_1,(s_2(0)\odot s_2(0))\vee s_2(0))=(1_1,0_2)\vee (0_1,0\vee s_2(0))\\
&=& (1_1,0_2)\vee (0_1,s_2(0))=(1_1,0_2)\vee s(0_1,0_2).
\end{eqnarray*}
Hence, $(1_1,0_2)$ satisfies the conditions of Lemma \ref{3.9} which means $\varphi^{-1}(1_1,0_2)$ satisfies the conditions, too.
Lemma \ref{3.9} implies that $\varphi^{-1}(1_1,0_2)=u$ and, consequently, $[0,u]\cong [(0_1,0_2),(1_1,0_2)]\cong M_1$.
Also, $u^\sim=\varphi^{-1}(1_1,0_2)^\sim=\varphi^{-1}(0_1,1_2)$, so $[0,u^\sim]\cong [(0_1,0_2),(0_1,1_2)]\cong M_2$.
\end{proof}

\begin{cor}\label{proper PMV}
Let $r:M\to M$ be a square root on a pseudo MV-algebra $(M;\oplus,^-,^\sim,0,1)$.
\begin{itemize}[nolistsep]
\item[{\rm (i)}] If $M$ does not meet the conditions {\rm (i)} and
{\rm (ii)} in Theorem {\rm \ref{3.10}}, then there exists a normal ideal $I$ on $M$ such that $M/I$ is a strict pseudo MV-algebra.
\item[{\rm (ii)}] If $I$ is a normal ideal of $M$ such that $M/I$ is strict, then $[0,u]\s I$, where $u=r(0)^-\odot r(0)^-$.
\end{itemize}
\end{cor}

\begin{proof}
(i) Consider the notations from the proof of  Theorem \ref{3.10}. Let $\pi_2:M_1\times M_2\to M_2$ sending $(x_1,y_1)$ to $y_1$ be the natural projection map.
Clearly, it is a homomorphism of pseudo MV-algebras. Set $I:=(\pi_2\circ \varphi)^{-1}(\{0\})$. Then $I$ is a normal ideal of $M$ and $M/I\cong M_2$.
From Theorem \ref{corHome} and the note just before Proposition \ref{strict-pro}, it follows that $M/I$ is a strict pseudo MV-algebra.

(ii) Let $I$ be a normal ideal of $M$ such that $M/I$ is a strict pseudo MV-algebra.
By Proposition \ref{Homo}, $r_I:M/I\to M/I$, defined by $r_I(x/I)=r(x)/I$, is a square root on $M/I$
(note that the natural projection map $f:M\to M/I$ sending $x$ to $x/I$ is onto), so $r_I(0/I)=r_I(0/I)^-=r(0)^{-}/I$.
It follows that $r(0)^-\odot r(0)^-\in I$, that is $[0,u]\s I$.
\end{proof}

\begin{cor}\label{DIPMVSQ}
Let $(M;\oplus,^-,^\sim,0,1)$ be a directly indecomposable pseudo MV-algebra with  square root. If $ |M|\ne 2$, then $M$ is strict.
\end{cor}

\begin{proof}
Let $M$ be a directly indecomposable pseudo MV-algebra.
If $M=\{0\}$, it is clearly strict. Let $2< |M|$ and let $u=r(0)^-\odot r(0)^-$. Then $u \in \B(M)$.
Since $M$ is directly indecomposable, in virtue of Remark \ref{3.6} and Proposition \ref{3.10}, $ \B(M)=\{0,1\}$.
If $u=1$, then $r(0)=0$, that is $M=\B(M)$, it contradicts $|M|>2$.
If $u=0$, then by the proof of Proposition \ref{3.10}, $r$ is strict.
\end{proof}

\begin{rmk}\label{3.15}
Let $\mathsf{Bool}$ be the class of Boolean algebras, $\mathsf{PMVSQ}$ be the class of pseudo MV-algebras with square roots, and $\mathsf{PMVSQ}_s$
be the class of pseudo MV-algebras with strict square roots. We know that the first and second classes are varieties.
For the third one, it suffices to consider the identities in Theorem \ref{variety} and an additional identity
$r(0)=r(0)^-$. Then we can easily show that $\mathsf{PMVSQ}_s$ is a variety, too.
Now, from Theorem \ref{3.10} we conclude that the variety $\mathcal{V}(\mathsf{Bool}\cup \mathsf{PMVSQ}_s)$, generated by $\mathsf{Bool}$ and $\mathsf{PMVSQ}_s$, is equal to $\mathsf{PMVSQ}$.
\end{rmk}

\begin{open}\label{op:1}
With respect to Proposition \ref{3.2}(8), show that the variety of pseudo MV-algebras with square roots properly contains the variety of pseudo MV-algebras with square roots satisfying $r(x)\odot r(y)\le r(x\odot y)$, $x,y\in M$.
\end{open}

Similarly to \cite[Page 16]{Bel}, we define a {\em strongly atomless} pseudo MV-algebra.
It is a pseudo MV-algebra $M$ such that for each non-zero element $x\in M\setminus\{0\}$, there exists a normal prime ideal
$P$ such that $x\notin P$ and $x/P$ is not an atom of $M/P$. Clearly, each strongly atomless pseudo MV-algebra is representable (since $\bigcap\{P\in \spec(M)\colon P \mbox{ is normal}\}=\{0\}$).
\begin{prop}\label{3.16}
Each strongly atomless pseudo MV-algebra is atomless.
\end{prop}

\begin{proof}
Let $(M;\oplus,^-,^\sim,0,1)$ be strongly atomless.
Let $a$ be an atom of $M$. Then there exists a normal prime ideal $P$ of $M$ such that $a\notin P$ and $a/P$ is not an atom of $M/P$.
Assume that $y$ is an arbitrary element of $M$ such that $0/P<y/P\leq a/P$. Then $a\odot y^-\in P$.
From $a\odot y^-\leq a$ and $a$ is an atom of $M$, we get that $a\odot y^-=a$ or $a\odot y^-=0$.

(1) If $a\odot y^-=a$, then $a^-\vee y^-=a^-\oplus (a\odot y^-)=a^-\oplus a=1$, which means $a\wedge y=0\in P$. Thus
$y\in P$ (since $a\notin P$), so $y/P=0/P$.

(2) If $a\odot y^-=0$, then $(a\odot y^-)\oplus (y\odot a^-)=y\odot a^-\in P$. It follows that $y/P=a/P$.

Now, (1) and (2) imply that $a/P$ is an atom of $M/P$, which is a contradiction. Therefore, $M$ is atomless.
\end{proof}

\begin{thm}\label{3.17}
Let $r$ be a square root on a pseudo MV-algebra $(M;\oplus,^-,^\sim,0,1)$. If $M$ is strongly atomless, then $r$ is strict.
\end{thm}

\begin{proof}
Let $M$ be strongly atomless. Set $u:=r(0)^-\odot r(0)^-$. By Proposition \ref{3.2}(11), $u$ is a Boolean element. We claim that $u=0$.
Otherwise, $u\neq 0$ implies that there is $P\in\spec(M)$ such that $u\notin P$ and $u/P$ is not an atom of the pseudo MV-algebra $M/P$.
Since $M/P$ is linearly ordered, and $u/P$ is a non-zero Boolean element of the pseudo MV-algebra $M/P$, then $u/P$ is the top element of
$M/P$ and so $M/P=[0/P,u/P]$. By the proof of Theorem \ref{3.10}, $M_1=[0,u]$ is a Boolean algebra. Let $0/P\le x/P\le u/P$ for some $x\in M$. Then $x\wedge u\in [0,u]$ implies that
$x\wedge u$ is a Boolean element of $M$. From $x/P=(x\wedge u)/P$, we conclude that $x/P$ is a Boolean element of the linearly ordered
pseudo MV-algebra $M/P$. Thus $x/P=0/P$ or $x/P=u/P$, that is, $u/P$ is an atom of $M/P$, which contradicts the assumption.
Therefore, $u=0$ and so $r(0)^-=r(0)$.
\end{proof}

\section{Representations of Square Roots on Representable Symmetric Pseudo MV-algebras}

Applying Theorem \ref{3.10}, we show how each square root on a representable and symmetric pseudo MV-algebra or each strict weak square root on a representable pseudo MV-algebra can be represented using the group addition of the representing unital $\ell$-group. At the end of the section, we use this result to characterize strict pseudo MV-algebras by the strong atomless property.

\begin{thm}\label{7.1}
Let $(M;\oplus,^-,^\sim,0,1)$ be a totally ordered symmetric pseudo MV-algebra with a square root $r:M\to M$ and let $(G,u)$ be its related unital $\ell$-group. If $r(0)=0$, then $r=\id_M$. If $r(0)>0$,
for each element $x\in M$, the element $(x+u)/2$ exists in $M$, and $r(x)=(x+u)/2$ for $x\in M$, where $+$ denotes the group addition in $G$.
\end{thm}

\begin{proof}
Since $M=\Gamma(G,u)$, we have $G$ is also totally ordered, so it enjoys unique extraction of roots.

(1) If $r(0)=0$, by Theorem \ref{3.5}, $M$ is the Boolean algebra $\{0,1\}$, and $r=\id_M$.

(2) Assume that $r(0)\neq 0$. For each $x\in M$, we have two cases:

(i) If $x\neq 0$, then $x=r(x)\odot r(x)=(r(x)-u+r(x))\vee 0=(2r(x)-u)$ because $M$ is totally ordered. Then $x=2r(x)-u$ and so $x+u=2r(x)$. The unique extraction of roots property implies that the element $(x+u)/2$ exists in $M$. Hence, for each $x\in M\setminus\{0\}$, $r(x)=(x+u)/2$.

(ii) If $x=0$, then $0=r(0)\odot r(0)=(2r(0)-u)\vee 0$ entails that $2r(0)-u\leq 0$, that is $2r(0)\leq u$. On the other hand, if
$y\in M$, we have $y\odot y\leq 0$ \iff $2y\leq u$, so by (Sq2), $2y\leq u$ implies that $y\leq r(0)$. That is
$r(0)=\max\{z\in M\colon 2z\leq u\}$. Also, if $g\in G$ be such that $2g\leq u$, then $g\in G^-$ implies that $g\leq r(0)$ and
from $g\in G^+$ we get $g\leq 2g\leq u$, that is $g\in [0,u]=M$ which means $g\leq r(0)$.
Hence
\begin{eqnarray}\label{e1458}
r(0)=\max\{z\in G\colon 2z\leq u\}.
\end{eqnarray}
By Proposition \ref{3.2}(11), $r(0)^-\odot r(0)^-\in \B(M)$, so $r(0)^-\odot r(0)^-=0$ or $r(0)^-\odot r(0)^-=1$ (since $M$ is totally ordered).
If $r(0)^-\odot r(0)^-=1$, then $r(0)^-=1$ entails that $r(0)=0$, but this is excluded. If $r(0)^-\odot r(0)^-=0$, then $r(0)\oplus r(0)=u$ entails
$r(0)+r(0)\geq u$ and so by (\ref{e1458}), $2r(0)=u$. In this case, $r(0)=(0+u)/2$.
\end{proof}

\begin{thm}\label{7.2}
Let $M=\Gamma(G,u)$ be a representable symmetric pseudo MV-algebra with a strict square root $r$, where $(G,u)$ is a unital $\ell$-group. Then $(x+u)/2$ exists for each $x\in M$, and
\begin{equation}\label{eq:r-strict}
r(x)=\frac{x+u}{2},\quad x\in M,
\end{equation}
where $+$ denotes the group addition in $G$.
\end{thm}

\begin{proof}
Assume that $X$ is the set of all normal prime ideals $P\ne M$ of $M$. The set $X$ is non-empty because $M$ is representable. We note that also $G$ is representable. Consider the embedding $\varphi:M\to M_0:=\prod_{P\in X}M/P$ defined by $\varphi(x)=(x/P)_{P\in X}$. We have $r(0)>0$.

(1) For each $P\in X$, we have $r(0)\notin P$ (since $r$ is strict and $u=r(0)\oplus r(0)= r(0)^-\oplus r(0)\in P$), consequently, $r(0)/P\neq 0/P$, and $r_P(x/P)=r(x)/P$, is a square root on the totally ordered pseudo MV-algebra $M/P$ and $r_P$ is strict.

(2) By Proposition \ref{Homo}, the homomorphism $\pi_P\circ\varphi:M\to M/P$ induces a strict square root $t_P:M/P\to M/P$ defined by $t_P(x/P)=r(x)/P$. Let $\hat P$ be the $\ell$-ideal (i.e. a normal convex $\ell$-subgroup of $G$) of $G$ generated by $P$. Then $\Gamma(G/\hat P,u/\hat P)$ is a totally ordered, symmetric pseudo MV-algebra such that $M/P\cong \Gamma(G/\hat P,u/\hat P)$.
Since $M/P$ is totally ordered, (1) implies that $M/P$ is not a Boolean algebra and so by Theorem \ref{7.1},  $r(x)/P=t_P(x/P)=(x/P+_Pu/P)/2$, where $+_P$ denotes the group addition in the unital $\ell$-group $(G/\hat P,u/\hat P)=\mathbf{\Psi}(M/P)$. If we set $t((x/P)_{P\in X})=(t_P(x/P))_{P\in X}$, then $t$ is a square root on $M_0$ and $r$ is the restriction of $t$ onto $M$.

(3) Consider the unital $\ell$-group $\mathbf{\Psi}(\prod_{P\in X}M/P)=\prod_{P\in X}(\mathbf{\Psi}(M/P))$. Due to Theorem \ref{functor}, the functor $\mathbf{\Psi}:\mathcal{PMV}\ra \mathcal{UG}$
induces an $\ell$-group homomorphism $\mathbf{\Psi}(\varphi):(G,u)\to \mathbf{\Psi}(\prod_{P\in X}M/P)$ preserving the strong unit that is injective.
Recall that $\varphi(g)=\mathbf{\Psi}(\varphi)(g)$ for all $g\in M$. We have
$\varphi(r(x))=(r(x)/P)_{P\in X}=(t_P(x/P))_{P\in X}=((x/P+_Pu/P)/2)_{P\in X}$. Also,
$((x/P+_Pu/P)/2)_{P\in X}+((x/P+_Pu/P)/2)_{P\in X}=(x/P+_Pu/P)_{P\in X}$, so $((x/P+_Pu/P)_{P\in X})/2$ exists and is equal to
$((x/P+_Pu/P)/2)_{P\in X}$.

(4) By (3), $\varphi(r(x))+\varphi(r(x))=(x/P+_P u/P)_{P\in X}=((x+_Pu)/P)_{P\in X}$, so
\begin{eqnarray*}
r(x)+r(x)&=&(\mathbf{\Psi}(\varphi))^{-1}(\varphi(r(x))+\varphi(r(x)))=
(\mathbf{\Psi}(\varphi))^{-1}(\varphi(x)+\varphi(x))\\
&=&\mathbf{\Psi}(\varphi)^{-1}(\mathbf{\Psi}(\varphi)(x+u))=x+u,
\end{eqnarray*}
which means $r(x)=(x+u)/2$.
\end{proof}

\begin{thm}\label{7.3}
Let $M=\Gamma(G,u)$ be a representable symmetric pseudo MV-algebra with a square root $r:M\to M$, where $(G,u)$ is a unital $\ell$-group. For each $x\in M$, $r(x)=(x\wedge w)\vee((x\wedge w^-)+w^-)/2$, where
$w=r(0)^-\odot r(0)^-$.
\end{thm}

\begin{proof}
By Proposition \ref{3.2}(11), $w$ is a Boolean element of $M$.
If $w=1$, then $M$ is a Boolean algebra, $w^-=0$, $r=\id_M$ and $(x\wedge w)\vee((x\wedge w^-)+w^-)/2=x=r(x)$ for all $x\in M$.

If $w=0$, then $r$ is strict, $w^-=1$ and so by Theorem \ref{7.1},
$r(x)=(x+1)/2= (x\wedge w)\vee((x\wedge w^-)+w^-)/2$.

Suppose that $w\notin\{0,1\}$ and set $v:=w^-$. By Theorem \ref{3.10}, $M\cong M_1\times M_2$, where
$M_1$ is the Boolean algebra $([0,w];\oplus,^{-w},^{\sim w},0,w)$ and $M_2$ is the strict pseudo MV-algebra $([0,w];\oplus,^{-v},^{\sim v},0,v)$.
Consider the homomorphisms $f_1:M\to M_1$ and $f_2:M\to M_2$ defined by $f_1(x)=x\wedge w$ and $f_2(y)=y\wedge w^-$, for each $x\in M$. Clearly, $f_1$ and $f_2$ are surjective maps. By Proposition \ref{Homo}, $t_1(f_1(x))=f_1(r(x))=r(x)\wedge w$ and
$t_2(f_2(x))=f_2(r(x))=r(x)\wedge w^-$ (for all $x\in M$) are square roots on $M_1$ and $M_2$, respectively.
For each $x\in M$, we have
\begin{eqnarray*}
r(x)\wedge w&=& f_1(r(x))=t_1(f_1(x))=f_1(x)=x\wedge w\quad \mbox{ by Theorem \ref{7.1}}\\
r(x)\wedge w'&=& f_2(r(x))=t_2(f_2(x))=(f_2(x)+w')/2=((x\wedge w')+w')/2\quad \mbox{ by Theorem \ref{7.2}}.
\end{eqnarray*}
It follows that for all $x\in M$
\begin{eqnarray*}
r(x)&=& r(x)\wedge (w\vee w^-)=(r(x)\wedge w)\vee (r(x)\wedge w^-)=(x\wedge w)\vee \big((x\wedge w^-)+w^-\big)/2.
\end{eqnarray*}
\end{proof}

Inspired by Theorems \ref{7.1}--\ref{7.2} and Example \ref{3.4}(v), we present the following characterization of strict weak square roots on representable pseudo MV-algebras, which are not necessarily symmetric.

\begin{thm}\label{th:wsr}
Let $r$ be a strict weak square root on a representable pseudo MV-algebra $M=\Gamma(G,u)$. Then $((x-u)/2)+u$ exists in $M$ for each $x\in M$ and $r(x)=((x-u)/2)+u$, $x\in M$, where $+$ denotes the group addition in the group $G$. Moreover, $M$ is two-divisible.
\end{thm}

\begin{proof}
The proof follows basic ideas from the proofs of Theorems \ref{7.1}--\ref{7.2}. Since $r$ is strict, $r(0)>0$.

(I) Let $M$ be totally ordered, and so is the group $G$. If we take $x\in M$, we have $x=r(x)\odot r(x)= (r(x)-u+r(x))\vee 0 = (2(r(x)-u)+u)\vee 0$. We have two cases:

Case (1). Let $x>0$. Since $M$ is totally ordered, $x= (2(r(x)-u)+u)\vee 0 = (2(r(x)-u)+u)$, so that $r(x)= ((x-u)/2)+u$.

Case (2). Let $x=0$. Then $u= r(0)^-\oplus r(0)=r(0)+r(0)=2r(0)$, that is, $r(0)=u/2= ((0-u)/2)+u$.

(II) Let $M$  be a subdirect product of totally ordered pseudo MV-algebras. We follow all steps of the proof of Theorem \ref{7.2}. Namely, let $X$ be the set of all normal prime ideals $P\ne M$ of $M$. The set $X$ is non-empty because $M$ is representable. Consider the embedding $\varphi:M\to M_0:=\prod_{P\in X}M/P$ defined by $\varphi(x)=(x/P)_{P\in X}$. Then we repeat all steps (1)--(4) from Theorem \ref{7.2}. We note that $r_P(x/P):=r(x)/P$ is a strict weak square root on the totally ordered pseudo MV-algebra $M/P$. In steps (2)--(4), we change variants of $(x+u)/2$ to the corresponding variants of $((x-u)/2)+u$, and we apply part (I) of the present proof, which finishes the statement in question.

Finally, we can establish that $M$ is two-divisible. Let $x\in M$, then $r(x^\sim)=(-x+u-u)/2+u= (-x)/2+u$. So that $(-x)/2$ exists in $G$. Then $-(-x/2)-(-x/2)=-((-x)/2)-((-x/2)=-(-x)=x$. Whence, $-(-x/2)=x/2$ exists in $M$ and $M$ is two divisible.
\end{proof}

We note that if $r$ is a strict weak square root on a pseudo MV-algebra $M=\Gamma(G,u)$, where $G$ enjoys unique extraction of roots, then $r(0)=u/2$, see Case (2) in the latter proof.

Cyclic elements of MV-algebras were studied in \cite{Tor}, and cyclic elements of pseudo MV-algebras in \cite{225}. We say that an element $a$ of a pseudo MV-algebra $M$ is said to be {\it cyclic of order} $n>0$ if $na$ exists in $M$ and $na
= 1$. If $a$ is cyclic of order $n$, due to associativity of the partial addition $+$, $a^-=a^\sim$. In fact, $a^- = (n-1)a = a^\sim$ because $a + (n - 1)a= 1 = (n - 1)a + a$. The set $\{0,a,2a,\ldots,na\}$ is a copy of the MV-algebra $\Gamma(\mathbb Z,1)$. Therefore, $M$ possesses a cyclic element of order $n\ge 1$ iff $M$ contains a copy of $\Gamma(\mathbb Z,n)$. If $M$ is representable and $a$ and $b$ are cyclic elements of order $n$, then $a=b$ due to the unique extraction of roots property of $M$. Whereas every MV-algebra has at most one cyclic element of order $n$, \cite[Thm 2.7]{Tor}, this is not true for pseudo MV-algebras as we can see from \cite[Ex 1]{225}. Theorem \ref{7.2} allows to characterize special cyclic elements of order $2^n$ for each $n\ge 1$ on a symmetric representable pseudo MV-algebras with square roots.

\begin{thm}\label{th:D}
Let $r$ be a strict square root on a symmetric representable pseudo MV-algebra $M=\Gamma(G,u)$. Then $M$ contains an isomorphic copy of the MV-algebra of dyadic numbers in the real interval $[0,1]$ and a unique cyclic element of order $2^n$ for each $n\ge 1$.

The same is true if $r$ is a strict weak square root on a representable pseudo MV-algebra $M$.
\end{thm}

\begin{proof}
By Theorem \ref{7.2}, $r(x)=(x+u)/2$ for each $x\in M$. Then $r(0)=u/2\in M$. By induction, we show that for each integer $n\ge 1$, the element $u/2^n$ is defined in $M$. Since $3u= 2(3\frac{u}{2})$, we have $\frac{3u}{2}=3\frac{u}{2}$. Moreover, if $a/2$ and $(a/2)/2$ exist in $G$ for some $a\in G$, then $a/4=(a/2)/2$.
We have $r(u/2)=(u/2+u)/2= 3/4u\in M$. But $3u/4-u/2=u/4\in M$. Assume $u/2^n\in M$. Then $r(u/2^n)=(u/2^n +u)/2= (1+2^n)u/2^{n+1}= u/2^{n+1}+ u/2\in M$ and $r(u/2^n)-u/2= u/2^{n+1}\in M$. Therefore, for each $n\ge 1$, $u/2^n\in M$ and $iu/2^n\in M$ for each $i=0,1,\ldots, 2^n$.

The element $u/2^n$ is a unique cyclic element of $M$ of order $2^n$, $n\ge 1$. The uniqueness follows from the unique extraction of roots on symmetric pseudo MV-algebras.

Now, let $r$ be a strict weak square root on a representable pseudo MV-algebra. Then $r(0)=u/2$, and we assert that $M$ contains each $u/2^n$, $n\ge 1$. Let $u/2^n\in M$. Then $r(u/2^n)=((u/2^n - u))/2 +u= u/2^{n+1}+u/2\in M$. Subtracting $u/2$, we get $u/2^{n+1}\in M$.
\end{proof}

The latter theorem is a generalization of an analogous statement from \cite[Thm 6.9]{Hol}, and \cite[Thm 2.4]{Amb} known for strict squares of roots on MV-algebras as well as of \cite[Thm 2.3]{Amb} for cyclic elements.

Now, we describe the existence of strict square roots on representable pseudo MV-algebras.

\begin{thm}\label{th:sym}
Let $M=\Gamma(G,u)$ be a representable pseudo MV-algebra with a strict square root $r$. Then $M$ is symmetric, two-divisible, and $u/2\in \C(G)$.
\end{thm}

\begin{proof}
(I) Assume that $M=\Gamma(G,u)$, where $(G,u)$ is a linearly ordered unital $\ell$-group.
If $M=\{0,1\}$, then $M$ is a Boolean algebra, and clearly, it is symmetric but $r=\id_M$ is not strict. Thus we have $|M|> 2$.
It suffices to show that $u+x=x+u$ for all $x\in M$. Let $x\in M$.
Corollary \ref{DIPMVSQ} implies that $r$ is strict and so $u-r(0)=r(0)^-=r(0)$, that is $u=2r(0)$ or equivalently, $r(0)=u/2$
(note that $G$ enjoys unique extraction of roots).
By \eqref{eq1-rmk-cor6}, $x\oplus u/2=u/2\oplus x$ consequently, $(x+u/2)\wedge u=(u/2+x)\wedge u$.

(i) If $x+u/2< u$, then $x<u/2$ and $u> x+u/2= x\oplus u/2= u/2 \oplus x = u/2+x$.

(ii) If $x+u/2=u$, that is $x=u/2$, then clearly $x+u/2=u/2+x$.

(iii) If $x+u/2> u$, then $u/2< x$, and $0\leq x-u/2< u/2$ and so by (i), $(x-u/2)+u/2=u/2+(x-u/2)$. It follows that
$x=u/2+(x-u/2)$ which entails $x+u/2=u/2+x$.

Now, (i)--(iii) imply that $x+u=u+x$ for all $x\in M$. Therefore, each linearly ordered pseudo MV-algebra with square root
is symmetric.

Now, we show that $M$ is two-divisible. By Theorem \ref{7.2}, $(x+u)/2$ exists in $M$ and $r(x)=(x+u)/2$ for each $x\in M$. If $x\in M$ is arbitrary, then $r(x^-)= (x^-+u)/2=(2u-x)/2$. Clearly, $(x-2u)/2$ is defined in $G$, so that $((x-2u)/2+u)+((x-2u)/2+u)=x$, which means $x/2$ exists in $G$ so that also in $M$.

(II) Let $M$ be a representable pseudo MV-algebra. As in the proof of Theorem \ref{7.2}, $M$ is a subdirect product of $\{M/P\mid P \in X\}$, where $X$ is the set of normal and prime ideals $P\ne M$ of $M$. Every $M/P$ is a linearly ordered pseudo MV-algebra and $r_P:M/P \to M/P$, defined by $r_P(x/P)=r(x/P)$, is a strict square root on $M/P$. By (I), $M/P$ is symmetric and two-divisible, so $M$ is also symmetric and two-divisible. In addition, $u/2\in \C(G)$.
\end{proof}

As a corollary of Theorem \ref{th:sym} with Example \ref{3.4}(iii), and Theorem \ref{th:wsr} with Example \ref{3.4}(v), we have the following result.

\begin{cor}\label{co:strict}
Let $M=\Gamma(G,u)$ be a representable pseudo MV-algebra.
\begin{itemize}
\item[{\rm (1)}] The pseudo MV-algebra $M$ has a strict square root if and only if $M$ is two-divisible and $u/2\in \C(G)$, so $M$ is symmetric, and $r(M)=[u/2,u]$.

\item[{\rm (2)}] The pseudo MV-algebra $M$ has a strict weak square root if and only if $M$ is two-divisible, then $r(M)=[u/2,u]$.

\item[{\rm (3)}] In Theorem {\rm \ref{7.2}} and Theorem {\rm \ref{7.3}}, the assumption ``$M$ is symmetric" is superfluous.

\item[{\rm (4)}] If $M$ is two-divisible and $u/2\in \C(G)$, then every strict weak square root on $M$ is a square root.
\end{itemize}
\end{cor}

\begin{proof}
(1) If $r$ is a strict square root on $M$, by Theorem \ref{th:sym}, $M$ is two divisible and $u/2\in \C(G)$. Conversely, let $M$ be two-divisible and $u/2\in \C(G)$. Then $M$ is symmetric and from $(x/2+u/2)+(x/2+u/2)= x/2+x/2+u/2+u/2=x+u$ we have $r(x):=(x+u)/2=x/2+u/2$ exists in $M$ for each $x$. As in Example \ref{3.4}(iii), it is possible to show that $r$ is a strict square root on $M$.

(2) If $r$ is a weak square root on $M$, by Theorem \ref{th:wsr}, $M$ is two-divisible. Conversely, let $M$ be two-divisible. Given $x\in M$, the element $r(x):=(x^-/2)^\sim= -((u-x)/2)+u= (x-u)/2+u$ is defined in $M$. As in Example \ref{3.4}(v), it is possible to show $r$ is a strict weak square root on $M$.
\end{proof}

\begin{prop}\label{pr:symm}
Let $(G,u)$ be a two-divisible unital $\ell$-group which enjoys unique extraction of roots, and if $x\in M=\Gamma(G,u)$, then $x/2\le u/2$. If $M$ has a strict square root, then $M$ is symmetric and $u/2\in C(G)$.
\end{prop}

\begin{proof}
Let $r:M\to M$ be a strict square root on $M$. From $r(0)=r(0)^-$ it follows that $2r(0)=u$.
Suppose that $x$ is an arbitrary element of $M$, by the assumptions, $x/2$ exists.
Due to Proposition \ref{rmk-cor}(vi),

\begin{eqnarray}\label{10103}
(r(0)+x/2)\wedge u=(r(0)\oplus x/2)=x/2\oplus r(0)=(x/2+r(0))\wedge u.
\end{eqnarray}
Whence, $r(0)+x/2,x/2+r(0)\leq r(0)+r(0)=u$, so by (\ref{10103}), $r(0)+x/2=x/2+r(0)$ which implies that
$u+x=r(0)+r(0)+x/2+x/2=x/2+x/2+r(0)+r(0)=x+u$. Thus $u,u/2\in \C(G)$ and $M$ is symmetric.
\end{proof}

We finish this section by applying Theorem \ref{7.2} to characterize strict pseudo MV-algebras by the strong atomless property.

In \cite[Thm 6.17]{Hol}, H{\"o}hle proved that if $M$ is a complete MV-algebra, then the concepts of a strict MV-algebra and a strongly atomless MV-algebra are equivalent.
In addition, each strongly atomless MV-algebra has a strict square root, see \cite[Lem 4.14]{Hol}.
We have generalized this result for pseudo MV-algebras in Theorem \ref{3.17}.

We are ready to show the converse for a more general case, not necessarily for complete MV-algebras.
Indeed, we show that if $(M;\oplus,^-,^\sim,0,1)$ is a representable symmetric pseudo MV-algebra with a strict square root $r:M\to M$, then $M$ is strongly atomless. It concludes that each strict MV-algebra is strongly atomless (note that each MV-algebra is representable and symmetric).

First, we establish the following helpful criterion.

\begin{lem}\label{3.17.1}
In each representable pseudo MV-algebra $(M;\oplus,^-,^\sim,0,1)$, the following conditions are equivalent:
\begin{itemize}[nolistsep]
\item[{\rm (i)}] $M$ is strongly atomless.
\item[{\rm (ii)}] For each $x\in M\setminus\{0\}$ there exists $y\in M$ such that
$0<y<x$ and $y\wedge (x\odot y^-)\neq 0$.
\item[{\rm (iii)}] For each $x\in M\setminus\{0\}$ there exists $y\in M$ such that
$0<y<x$ and $y\wedge (y^\sim\odot x)\neq 0$.
\end{itemize}
\end{lem}

\begin{proof}
We only prove (i) $\Leftrightarrow$ (ii), because the proof of (i) $\Leftrightarrow$ (iii) is similar.

Assume that $M$ is strongly atomless. Let there be $x\in M\setminus\{0\}$ such that for each $y\in M$ with $0<y<x$, we have $y\wedge (x\odot y^-)=0$. Choose a normal prime ideal $P$ of $M$ such that $x\notin P$. We claim that $x/P$ is an atom of $M/P$.
Let $0/P<z/P\leq x/P$ for some $z\in M$. Without loss of generality, we can assume that $z\leq x$ (since $z/P=z/P\wedge x/P=(x\wedge z)/P$).
By the assumption, $z\wedge (x\odot z^-)=0\in P$ and so $z\in P$ or $x\odot z^-\in P$.
If $z\in P$, then $z/P=0/P$. If $x\odot z^-\in P$, then $x/P\odot (z/P)^-=0/P$, that is $x/P\leq z/P$. Thus $x/P=z/P$, and $x/P$ is an atom of $M/P$, which contradicts (i).

Now, let (ii) hold. Choose $x\in M\setminus\{0\}$. Then there exists $y\in M$ such that
$0<y<x$ and $y\wedge (x\odot y^-)\neq 0$. Let $P$ be a normal prime ideal of $M$ such that $y\wedge (x\odot y^-)\notin P$.
Then $0/P<y/P\leq x/P$. If $y/P=x/P$, then $y\odot x^-,x\odot y^-\in P$ and so $y\wedge (x\odot y^-)\in P$, which is a contradiction.
Hence, $y/P<x/P$ and so $x/P$ is not an atom of $M/P$. Therefore, $M$ is strongly atomless.
\end{proof}

\begin{thm}\label{8.1}
Let $(M;\oplus,^-,^\sim,0,1)$ be a representable pseudo MV-algebra. If $M$ is strict, then it is strongly atomless.
\end{thm}

\begin{proof}
By Theorem \ref{th:sym}, $M$ is symmetric. Let $r:M\to M$ be a strict square root. Due to Theorem \ref{7.2}, $r(x)=(x+u)/2$ for each $x\in M$.
By Lemma \ref{3.17.1}, it suffices to show that for each $x\in M\setminus\{0\}$, there exists $y\in M$
such that $0<y<x$ and $y\wedge (x\odot y^-)\neq 0$. If $|M|=1$, then $M=\{0\}$ and so the proof is clear. Let $2\leq |M|$.
Choose $x\in M\setminus\{0\}$.

(1) If $x=1$, then choose $y\in M\setminus \B(M)$ such that $0<y<x$ (since $M$ is strict, by Theorem \ref{3.10}, $M$ is not a Boolean algebra and so we can find such an element). Then $y\wedge (x\odot y^-)=y\wedge y^->0$, since $y$ is not a Boolean element (see \cite[4.2]{georgescu}).

(2) If $x\in M\setminus\{0,1\}$ is such that $x\in \B(M)$, then by Proposition \ref{rmk-cor}(i), $r(x)=x\vee r(0)$ and $r(x^-)=x^-\vee r(0)$. Set $y:=r(x^-)^\sim$.
From relation $(\ref{eq1-rmk-cor3})$, we know that $0\leq y\leq x$. If $y=0$, then $r(x^-)=1=r(1)$ and so $x=0$, which is a contradiction.
If $y=x$, in view of Proposition \ref{3.2}(1), $x^-\vee r(0)=r(x^-)=x^-$. It follows that $x\leq r(0)^-=r(0)$ entails that $x\odot x\leq r(0)\odot r(0)=0$, which contradicts the assumption. Thus, $0<y<x$. We have
\begin{eqnarray*}
y\wedge (x\odot y^-)&=& r(x^-)^\sim \wedge (x\wedge r(x^-))=r(x^-)^\sim\wedge r(x^-)\quad \mbox{ by \eqref{eq1-rmk-cor3},  $r(x^-)^\sim\leq x$}\\
&=& (x\wedge r(0)^\sim) \wedge (x^-\vee r(0))=(x\wedge r(0)) \wedge (x^-\vee r(0))\quad \mbox{ since $r$ is strict}\\
&=& x\wedge r(0).
\end{eqnarray*}
We claim that $x\wedge r(0)\neq 0$, otherwise, since $x\in\B(M)$, $r(0)\leq x^-$ and so $x\leq r(0)^-=r(0)$. It follows that
$x=x\odot x\leq r(0)\odot r(0)=0$, which is absurd.

(3) If $x\in M\setminus \B(M)$, then similarly to (2), we can show that $0<r(x^-)^\sim<x$. Note that by Proposition \ref{rmk-cor}(i), $x\vee r(0)\leq r(x)$.
Set $y:=r(x^-)^\sim$. We will show that $0\neq y\wedge (x\odot y^-)=r(x^-)^\sim\wedge (x\odot r(x^-))$.
Let $M=\Gamma(G,u)$, where $(G,u)$ is a unital $\ell$-group.
By Theorem \ref{7.2}, $r(x^-)=(x^-+u)/2=(2u-x)/2$.

(i) Suppose that $a$ is an arbitrary element of $G$ such that $a/2$ exists, then
$a/2+u=(a+2u)/2$. Indeed, $(a/2+u)+(a/2+u)=a/2+a/2+2u=a+2u$, so $a/2+u=(a+2u)/2$ (note that $G$ enjoys unique extraction of roots).
Also, $(a/2-u)+(a/2-u)=a/2+a/2-2u=a-2u$, so $a/2-u=(a-2u)/2$.

(ii) For each $x\in G$, if $x/2$ exists, then $(-x)/2$ exists, too and $(-x)/2=-x/2$.
Because $x/2+x/2=x$ and so $-x=-x/2-x/2$.

(iii) Let $a\in M$. On the $\ell$-group $G$  by (i), we have
$r(a^-)^\sim=-((2u-a)/2)+u=(a-2u)/2+u=(a-2u+2u)/2=a/2$. That is, $a/2$ exists and belongs to $M$, see also Corollary \ref{co:strict}(1).

\begin{eqnarray*}
\label{ghghty12} r(x^-)^\sim\wedge (x\odot r(x^-))&=& x/2\wedge
\big((x+ (2u-x)/2-u)\vee 0\big)\quad \mbox{ by (i) and (iii)}\\
&=& x/2\wedge \big(x+ (2u-x)/2-u)\big)\\
&=& x/2\wedge \big(x+ (2u-x-2u)/2\big)=x/2\wedge \big(x+ (-x)/2\big) \quad \mbox{ by (i)}\\
&=& x/2\wedge x/2=x/2\quad \mbox{ by (ii).}
\end{eqnarray*}
Since $x\neq 0$, if we put $y=r(x^-)^\sim$, then $y\wedge (x\odot y^-)=x/2\neq 0$.

From (1)--(3) and Lemma \ref{3.17.1}, it follows that $M$ is strongly atomless.
\end{proof}

\begin{open}\label{op:2}
It would be interesting to find a representation of strict square roots on pseudo MV-algebras that are neither not necessarily representable nor symmetric in the way of the last theorems.
\end{open}

\section{Examples of Weak Square Roots that are Not Square Roots}

In the section, we present two examples of non-symmetric  representable pseudo MV-algebras that have weak square roots but no
square root.

Let $(H,u)$ be a unital $\ell$-group and $G$ be an $\ell$-group, both groups written additively. Let $\Aut(G)$ be the group of automorphisms of the po-group $G$. For example, given $g_0\in G$, the mapping $\psi_{g_0}:G\to G$, defined by $\psi_{g_0}(g)=-g_0+g+g_0$, $g \in G$, is an automorphism of $G$. Take a group homomorphism $\phi:H\to \Aut(G)$. The automorphism $\phi(h): G\to G$ is denoted simply by $\phi_h$ for each $h\in H$. Then $\phi_{h_1+h_2}= \phi_{h_1}\circ \phi_{h_2}$ for all $h_1,h_2\in H$, and $\phi_0$ is the identity on $G$. We define a {\it semidirect product} concerning $\phi$ of $H$ and $G$, $H\semid G$, as a special group structure on the direct product $H\times G$ ordered lexicographically, where the group operations on $H\semid G$ are defined by
$$
(h_1,g_1)\cdot (h_2,g_2)=(h_1+h_2,\phi_{h_2}(g_1)+g_2),\quad (h,g)^{-1}=(-h,\phi_{-h}(-g)).
$$
The neutral element is $(0_H,0_G)$, and $(u,0_G)$ is a strong unit of $H\semid G$. Then $H\semid G$ is an $\ell$-group if and only if $H$ is linearly ordered.

We define a pseudo MV-algebra $\Gamma(H,u)$ and a pseudo MV-algebra
\begin{equation}\label{eq:3.1}
M^{\phi}_{H,u}(G)=\Gamma(H\semid G,(u,0_G)).
\end{equation}

Then in the pseudo MV-algebra $M^{\phi}_{H,u}(G)$, we have
$$
(h,g)^-=(h^-,-\phi_{-h}(g)),\quad (h,g)^\sim = (h^\sim,-\phi_{h^\sim}(g)),
$$
and $M^{\phi}_{H,u}(G)$ is symmetric iff $\phi_u$ is the identity on $G$.

In the class of such pseudo MV-algebras, we find an example of a strict weak square root that is not a square root on a non-symmetric pseudo MV-algebra.

Let $H=(0,\infty)$ be the group of positive real numbers with the natural product and the natural order of reals numbers, and with the neutral element $1$. The element $u=2$ is a strong unit for $H$. On the other hand, let $G=\mathbb R$ with the natural addition of real numbers. For each $h>0$, the mapping $\phi_h:\mathbb R\to \mathbb R$, defined by $\phi_h(g)=hg$, $g \in \mathbb R$, is a group automorphism of $\mathbb R$. Define the pseudo MV-algebra $M=M^{\phi}_{H,2}(G)$, its top element is $(2,0)$. It is totally ordered and not symmetric: We have
$$
(h,g)^{-1}=(\frac{1}{h},-\frac{1}{h}g),\quad (h,g)^-=(\frac{2}{h}, -\frac{1}{h}g), \quad (h,g)^\sim= (\frac{2}{h},-\frac{2}{h}g).
$$

For all $(h_1,g_1), (h_2,g_2)\in M$, we have
\begin{equation}\label{eq:odot}
(h_1,g_1)\odot (h_2,g_2)=\big(\frac{h_1h_2}{2}, \frac{h_2}{2} g_1+g_2\big)\vee (1,0).
\end{equation}

\begin{exm}\label{example}
Define a mapping $r:M\to M$ by
\begin{equation}\label{eq:square}
r(h_0,g_0)=\big(\sqrt{2h_0},\frac{2}{\sqrt{2h_0}+2}g_0\big),\quad (h_0,g_0)\in M.
\end{equation}
We assert that $r$ is a strict weak square root on $M$ of the form \eqref{eq:form2} that is not a square root, and no variant of the formula \eqref{eq:r-strict} does hold.
\end{exm}

\begin{proof}
Property (Sq1) follows from \eqref{eq:odot} and from \eqref{eq:square} we conclude (Sq2). The mapping $r$ is non-decreasing and $r(2,0)=(2,0)$. We show (Sq3) fails because
\begin{eqnarray*}
r((h,g)^-)&=& \big(\frac{2}{\sqrt{h}}, \frac{-1}{\sqrt{h}+h} g\big)\ne r(h,g)^-\oplus r(1,0)= \big(\frac{2}{\sqrt{h}},\frac{-\sqrt{2}}{h+\sqrt{2h}}g\big)\wedge (2,0),\\
r((h,g)^\sim)&=& \big(\frac{2}{\sqrt{h}}, \frac{-2}{\sqrt{h}+h}g\big)\ne r(1,0)\oplus r(h,g)^\sim =
\big(\frac{2}{\sqrt{h}},\frac{-2}{h+\sqrt{2h}}g\big)\wedge (2,0).
\end{eqnarray*}

Moreover, $r$ is strict while $r(1,0)=(\sqrt{2},0)= r(1,0)^-$. Since $H\semid \mathbb R$ is two-divisible, we can count $(h,g)/2$ which is
$$
(h,g)/2= \big(\sqrt{h},\frac{1}{\sqrt{h}+1}g\big).
$$
Therefore, $((h,g)+(2,0))/2= (2h,2g)$ and $(2,0)+(h,g)=(2h,g)$ exists in $M$ but
$$
((h,g)+(2,0))/2 = (\sqrt{2h}, (2/(\sqrt{2h}+1))g\ne r(h,g)\ne (\sqrt{2h}, (2/(\sqrt{2h}))g),
$$
so that for our weak square root $r$, we have no variant of the formula \eqref{eq:r-strict}, but it coincides with formula \eqref{eq:form2}. We note that $r$ is not standard while
\begin{eqnarray*}
r(h,g)\odot r(1,0)&=&(\sqrt{2h},(\sqrt{2}/(\sqrt{2h}+2))g)=(\sqrt{h}, (\sqrt{2}/(\sqrt{2h}+2))g)\\
&\ne& (\sqrt{h}, (2/(\sqrt{2h}+2))g)= r(1,0)\odot r(h,g).
\end{eqnarray*}
\end{proof}

As a direct corollary of the latter example and Theorem \ref{variety}, we have that the class of pseudo MV-algebras with square roots is a proper subvariety of the variety of pseudo MV-algebras with weak square roots. In addition, on a non-symmetric representable pseudo MV-algebra there is no square root, see Theorem \ref{th:sym}, but a weak square root can exist. Therefore, Example \ref{ex:sym}(ii) also has a weak square root but no square root.

Now, we present another example of a non-symmetric linearly ordered two-divisible pseudo MV-algebra having a weak square root but no square root.

\begin{exm}\label{example1}
Let $G_2=\mathbb R^2$ and let the group operations be defined as follows $(x_1,y_1)+(x_2,y_2)=(x_1+x_2,\mathrm{e}^{x_2}y_1+y_2)$, $-(x,y)= (-x,-\mathrm{e}^{-x} y)$, and with the neutral element $(0,0)$. We endow $G_2$ with the lexicographic order. Then $G_2$ is linearly ordered and $u_2=(1,0)$ is a strong unit of $G_2$. Moreover, $G_2$ is two-divisible with $(x,y)/2= (x/2,y/(\mathrm{e}^{x/2}+1))$. Then $M_2=\Gamma(G_2,u_2)$ is a non-symmetric pseudo MV-algebra while
$(x,y)^-=(1-x,-e^{-x}y)$, $(x,y)^\sim= (1-x,-e^{-x+1}y)$,
and $r(x,y)=((x,y)-(1,0))/2+(1,0)= (\frac{x+1}{2},\frac{y}{\mathrm{e}^{(x-1)/2}+1})$ is a weak square root due to Theorem \ref{th:sym}, $M$ has no square root.

The element $u_3=(\ln 2,0)$ is also a strong unit of $G_2$ and the unital $\ell$-group $(G,u)$ from Example \ref{example} and the unital $\ell$-group $(G_2,u_3)$ are isomorphic as unital groups, the unital group isomorphism $\psi:(G,u)\to (G_2,u_3)$ is given by $\psi(h,g)=(\ln h,g)$, $(h,g)\in G$. Therefore, the pseudo MV-algebras $M=\Gamma(G,u)$ and $M_3=\Gamma(G_2,u_3)$ are isomorphic, and $M_3$ has a strict weak square root $r_3$ given by \eqref{eq:form2} that is not a square root. Moreover, $\psi(r(h,g))=r_3(\psi(h,g))$, $(h,g)\in M$.
\end{exm}

\begin{open}\label{op:3}
Characterize weak square roots.
\end{open}

\section{Conclusion}

To generalize the notion of square roots for pseudo MV-algebras from ones on MV-algebras in the way by \cite{Hol}, it was necessary to introduce square roots and weak square roots. They on MV-algebras coincide, but for pseudo MV-algebras, they can be different; see Examples \ref{example}--\ref{example1}. In the paper, we presented the results of our study of square roots. We have characterized Boolean algebras among pseudo MV-algebras using square roots, Theorem \ref{3.5}. We defined strict square roots and used them to characterize square roots, we showed that a pseudo MV-algebra $M$ with square root is either Boolean or strict or isomorphic to the direct product of a Boolean algebra and a strict pseudo MV-algebra, see Theorem \ref{3.10}. The class of pseudo MV-algebras with square root forms a proper subvariety of pseudo MV-algebras. It properly contains the variety of pseudo MV-algebras with strict square root, Theorem \ref{variety}. We introduced strongly atomless pseudo MV-algebras to characterize strict pseudo MV-algebras. We showed that if $M$ has a strict square root, then it is strongly atomless, Theorem \ref{3.17}, and the converse also holds for symmetric representable pseudo MV-algebras, see Theorem \ref{8.1}. Theorem \ref{7.2} and Theorem \ref{7.3} characterize every strict square root and every square root, respectively, on a representable symmetric pseudo MV-algebra using group addition in the corresponding unital $\ell$-group. An analogous result was established for strict weak square root on a representable pseudo MV-algebra, Theorem \ref{th:wsr}. We also showed that on a non-symmetric representable pseudo MV-algebra, there is no square root, Theorem \ref{th:sym}, but a weak square root can exist.

The study of square roots presented interesting research. It would be valuable to extend our study to strict square roots and strict weak square roots on not necessarily representable and not necessarily symmetric pseudo MV-algebras. We hope to completely characterize weak square roots.

\section{Acknowledgement}
The paper acknowledges the support by the grant of
the Slovak Research and Development Agency under contract APVV-20-0069  and the grant VEGA No. 2/0142/20 SAV,  A.D.

The project was also funded by the European Union's Horizon 2020 Research and Innovation Programme based on the Grant Agreement under the Marie Sk\l odowska-Curie funding scheme No. 945478 - SASPRO 2, project 1048/01/01,  O.Z.



\begin{thebibliography}{AnCo}


\bibitem[Amb]{Amb} R. Ambrosio, {\it Strict MV-algebras}, J. Math. Anal. Appl.
 {\bf 237} (1999), 320--326.\\ https://doi.org/10.1006/jmaa.1999.6482

\bibitem[Bel]{Bel} L.P. Belluce, {\it $\alpha$-Complete MV-algebras}, In: U. H{\"o}hle., E.P. Klement (eds), Non-Classical Logics and their Applications to Fuzzy Subsets: A Handbook of the Mathematical Foundations of Fuzzy Set Theory, Vol {\bf 32}, Springer, Dordrecht, pp. 7--21, 1995.
https://doi.org/10.1007/978-94-011-0215-5\_2

\bibitem[B{\v{e}}l]{Be} R. B{\v{e}}lohl{\'a}vek, {\it Some properties of residuated lattices}, Czechoslovak Math. J. {\bf 53} (2003), 161--171. https://doi.org/10.1023/A:1022935811257



\bibitem[Cha1]{Cha1}
C.C. Chang, {\it Algebraic analysis of many-valued logics}, Trans. Amer. Math. Soc. {\bf 88} (1958),
467--490. https://doi.org/10.2307/1993227

\bibitem[Cha2]{Cha2}
C.C. Chang, {\it A new proof of the completeness of the \L ukasiewicz axioms}, Trans. Amer. Math. Soc. {\bf 93} (1959), 74--80.
https://doi.org/10.2307/1993423

\bibitem[CDM]{CDM}
R. Cignoli, I.M.L. D'Ottaviano and D. Mundici, {\it ``Algebraic Foundations of Many-Valued Reasoning"}, Springer Science and Business Media, Dordrecht, 2000. https://doi.org/10.1007/978-94-015-9480-6

\bibitem[Dar]{Dar} M. Darnel, {\it ``Theory of Lattice-Ordered Groups"}, {\it CRC Press}, USA, 1994.

\bibitem[DGI]{DGI}
A. Di Nola, G. Georgescu, A. Iorgulescu, {\it Pseudo BL-algebras, Part I}, Multiple Valued Logic {\bf 8} (2002),  673--716.

\bibitem[Dvu1]{Dvu1}
A. Dvure\v censkij, {\it On pseudo MV-algebras}, Soft Computing {\bf 5} (2001), 347--354. \\
https://doi.org/10.1007/s005000100136

\bibitem[Dvu2]{DvuS}
A. Dvure\v censkij, {\it States on pseudo MV-algebras},  Studia Logica {\bf 68}  (2001), 301--327. \\
https://doi.org/10.1023/A:1012490620450

\bibitem[Dvu3]{Dvu2}
A. Dvure\v censkij, {\it Pseudo MV-algebras are intervals in $\ell$-groups}, J. Austral. Math. Soc. {\bf 72} (2002), 427--445.
https://doi.org/10.1017/S1446788700036806

\bibitem[Dvu4]{225}
A. Dvure\v censkij, {\it Cyclic elements and subalgebras of
GMV-algebras}, Soft Computing {\bf 14} (2010), 257--264.
DOI: 10.1007/s00500-009-0400-x

\bibitem[DvZa]{331}
A. Dvure\v{c}enskij, O. Zahiri, {\it On EMV-algebras with square roots}, submitted. http://arxiv.org/abs/2210.00862

\bibitem[Fuc]{Fuc} L. Fuchs, {\it Partially Ordered Algebraic Systems}, Pergamon Press, Oxford--New York, 1963.

\bibitem[GaTs]{GaTs}
N. Galatos, C. Tsinakis, {\it Generalized MV-algebras}, Fuzzy Sets and Systems {\bf 283} (2005), 254--291.
https://doi.org/10.1016/j.jalgebra.2004.07.002

\bibitem[GeIo]{georgescu}
G. Georgescu and A. Iorgulescu, {\it Pseudo MV-algebras}, Multiple-Valued Logics {\bf 6} (2001), 193--215.

\bibitem[Gla]{Gla} A.M.W. Glass,  {\it ``Partially Ordered Groups"}, World Scientific, Singapore, 1999.
https://doi.org/10.1142/3811

\bibitem[Haj]{Haj} P. H\'ajek, {\it Fuzzy logics with noncommutative conjuctions}, J. Logic Computation {\bf 13} (2003),
469--479. https://doi.org/10.1093/logcom/13.4.469

\bibitem[H{\"o}l]{Hol} U. H{\"o}hle, {\it Commutative, residuated $\ell$-monoids}. In: U. H{\"o}hle., E.P. Klement (eds), Non-Classical Logics and their Applications to Fuzzy Subsets: A Handbook of the Mathematical Foundations of Fuzzy Set Theory, Vol {\bf 32}, pp. 53--106. Springer, Dordrecht, 1995.
https://doi.org/10.1007/978-94-011-0215-5\_5

\bibitem[KoMe]{KoMe}
V.M. Kopytov, N.Ya. Medvedev, {\it ``The Theory of Lattice-Ordered Groups"}, Kluwer Academic Publication, Dordrecht, 1994. https://doi.org/10.1007/978-94-015-8304-6

\bibitem[Mun]{Mun}
D. Mundici, {\it Interpretation of AF C$^*$-algebras in \L ukasiewicz sentential calculus}, J. Funct. Anal. {\bf 65} (1986), 15--63.
https://doi.org/10.1016/0022-1236(86)90015-7

\bibitem[NPM]{NPM}
V. Nov\'ak, I. Perfilieva, J. Mo\v{c}ko\v{r}, {\it ``Mathematical Principles of Fuzzy Logic"},  Springer Science Business Media, New York, 1999. https://doi.org/10.1007/978-1-4615-5217-8

\bibitem[Rac]{Rac}
J. Rach\r unek, {\it A non-commutative generalization of MV-algebras}, Czechoslovak Math. J. {\bf 52} (2002), 255--273.
https://doi.org/10.1023/A:1021766309509

\bibitem[Tor]{Tor}
A. Torrens, {\it Cyclic elements in MV-algebras and Post algebras}, Math. Logic Quart. {\bf 40} (1994), 431--444. https://doi.org/10.1002/malq.19940400402

\bibitem[YaRu]{YaRu}
Y.Yang, W. Rump, {\it Pseudo-MV algebras as L-algebras}, Journal of Multiple-Valued Logic \& Soft Computing {\bf 19} (2012), 621--632.




\end{thebibliography}
\end{document}